\documentclass[12pt]{amsart} 
\usepackage[colorlinks=true, linkcolor=blue, citecolor=red,urlcolor=blue]{hyperref} 
\usepackage{amscd,amssymb,mathrsfs}
\usepackage{mathtools}
\usepackage{graphics}
\usepackage[dvipdf]{graphicx}
\usepackage{xypic}
\usepackage{subfigure}

\newtheorem{thm}{Theorem}[subsection]
\newtheorem{lemma}[thm]{Lemma}

\newtheorem{prop}[thm]{Proposition}

\newtheorem{defn}[thm]{Definition}
\newtheorem*{defn*}{Definition}

\begin{document}

\newcommand{\R}{\mathbb R}         
\newcommand{\C}{\mathbb C}                
\newcommand{\Z}{\mathbb Z}                
\newcommand{\Q}{\mathbb{Q}}
\newcommand{\Rplus}{\R^+}            

\renewcommand{\o}{\mathsf}	 
\newcommand{\Axis}{\o{Axis}}         
\newcommand{\Aut}{\o{Aut}}       
\newcommand{\Basis}{\o{Basis}}      
\renewcommand{\det}{\o{det}}          
\newcommand{\Det}{\o{Det}}          
\renewcommand{\iint}{\o{int}}	     
\newcommand{\Fix}{\o{Fix}}     

\newcommand{\Hom}{\o{Hom}}     
\newcommand{\Homeo}{\o{Homeo}} 
\newcommand{\Ht}{\o{H}^2}           
\renewcommand{\iint}{\o{int}}	     
\newcommand{\Ker}{\o{Ker}}		 
\renewcommand{\mod}{\o{mod}}	
\newcommand{\Mod}{\o{PGL(2,\Z)}}  
\renewcommand{\P}{\mathsf P}         
\newcommand{\rank}{\o{rank}}	
\newcommand{\sgn}{\o{sign}}	
\newcommand{\spn}{\o{span}}	
\newcommand{\T}{\mathsf T}          
\newcommand{\tr}{\o{tr}}                 

\newcommand{\Qpo}{\Q\o{P}^1}    

\newcommand{\MCG}{\o{Mod}(\surf)}  
\newcommand{\GLtZ}{\o{GL(2,\Z)}}  
\newcommand{\Ft}{\o{F}_2}        

\newcommand{\Prim}{\o{Prim}}   

\newcommand{\Out}{\o{Out}}       
\newcommand{\Inn}{\o{Inn}}         
\newcommand{\Ou}{\Out(\Ft)}      
\newcommand{\Au}{\Aut(\Ft)}      
\newcommand{\Sss}{\mathscr{S}} 

\newcommand{\Oto}{\o{O}(2,1)}
\newcommand{\SOoto}{{\o{SO}(2,1)^0}}
\newcommand{\SOto}{{\o{SO}(2,1)}}
\newcommand{\poto}{{\o{PO}(2,1)}}
\newcommand{\SLt}{\o{SL}(2)}
\newcommand{\SLtC}{\o{SL}(2,{\C})}
\newcommand{\PSLtC}{\o{PSL}(2,\C)}
\newcommand{\SLtR}{\o{SL}(2,\R)}
\newcommand{\PGLtR}{\o{PGL}(2,\R)}

\newcommand{\e}{\mathfrak{e}}          
\newcommand{\Fte}{\Ft^\e}               
\newcommand{\X}{\o{X}(\Ft,\SLtC)} 
\newcommand{\Inne}{\o{Inn}^\e}   

\newcommand{\E}{\o{E}}                
\newcommand{\V}{\o{V}}               
\newcommand{\Iso}{\o{Isom}^+(\E)}
\newcommand{\Sim}{\o{Sim}^+(\E)}
\newcommand{\Future}{\o{Future}}
\newcommand{\Past}{\o{Past}}

\newcommand{\surf}{\Sigma}           
\newcommand{\Fricke}{\mathfrak F}    
\newcommand{\FS}{{\Fricke}(\surf)}	
\newcommand{\Soo}{\surf_{1,1}}         
\newcommand{\Coo}{C_{1,1}}         
\newcommand{\SymThree}{\mathfrak{S}_3} 

\newcommand{\QQ}{\mathcal{Q}}           
\newcommand{\xii}{\xi^{-1}}

\newcommand{\Hh}{\mathfrak{h}}  
\newcommand{\HH}{\mathcal H}  
\newcommand{\CP}{{\partial\HH}} 
\newcommand{\iot}[1]{\iota_#1}     
\newcommand{\tiot}[1]{\widetilde\iota_#1}     

\newcommand{\vn}{\mathbf n} 
\newcommand{\vu}{\mathbf u}     
\newcommand{\vv}{\mathbf v}     
\newcommand{\vw}{\mathbf w}    
\newcommand{\vx}{\mathbf x}     
\newcommand{\va}{\mathbf a}     
\newcommand{\vb}{\mathbf b}    
\newcommand{\vy}{\mathbf y}    
\newcommand{\vc}{\mathbf c}    
\newcommand{\vt}{\mathbf{t}} 
\newcommand{\vs}{\mathbf{s}}       
\newcommand{\s}[1]{\vs_#1}          
\renewcommand{\v}[1]{\vv_#1}       
\renewcommand{\t}[1]{\vt_#1}         
\newcommand{\h}[1]{\Hh_#1}         
\renewcommand{\H}[1]{\HH_#1}     
\newcommand{\qv}{\mathbf{q}}       

\newcommand{\Ao}{\va}
\newcommand{\Bo}{\vb}
\newcommand{\Co}{\vc}

\newcommand{\xo}[1]{{#1}^0}
\newcommand{\xp}[1]{{#1}^+}
\newcommand{\xm}[1]{{#1}^-}
\newcommand{\xpm}[1]{{#1}^{\pm}}

\newcommand{\am}{A^{-1}} 
\newcommand{\bm}{B^{-1}} 

\newcommand{\HA}{\Hh_A}		
\newcommand{\HB}{\Hh_B}
\newcommand{\HAM}{A^{-1}(\HA^c)}
\newcommand{\HBM}{B(\HB^c)}
\newcommand{\HO}{\Hh_C}
\newcommand{\HI}{\Hh_i}

\renewcommand{\aa}{\mathfrak{a}} 
\newcommand{\Aa}{\mathscr{A}}  
\newcommand{\Tt}{\mathscr{T}} 

\newcommand{\sa}{\sqrt{a^2-4}}
\renewcommand{\sb}{\sqrt{b^2-4}}
\renewcommand{\sc}{\sqrt{c^2-4}}   

\newcommand{\ideal}{\mathcal T}	
\newcommand{\iq}{{\mathcal Q}}	


\newcommand{\ZZ}{Z^1(\Gamma,\V)}
\newcommand{\Ho}{H^1(\Ft,\V)} 
\newcommand{\ho}{H^1(\surf_0,\V)}	 

\newcommand{\cu}{\mathsf u} 
\newcommand{\cub}[1]{{\mathsf u}^{#1}} 

\newcommand{\PHH}{\Delta^+}
\newcommand{\PT}[1]{\Pi(#1)} 
\newcommand{\lcross}{\boxtimes}
\newcommand{\pxo}[1]{\left[ #1\right]} 

\newcommand{\Br}{\mathcal{B}} 
\newcommand{\idealB}{\ideal_\Br}      %

\newcommand{\Proper}{\o{Proper}}
\newcommand{\Tile}{\o{Tile}}
\newcommand{\Tame}{\o{Tame}}
\newcommand{\bb}{\mathfrak{b}}  
\newcommand{\BB}{\Basis(\Ft)}	
\newcommand{\BBB}{\mathscr{B}(\Ft)}	

\newcommand{\FF}{\mathscr{F}}  
\newcommand{\PP}{\mathscr{P}}  
\newcommand{\CC}{\Corner}  

\newcommand{\Abelianize}{\mathscr{A}} 

\newcommand{\Id}{\mathbb{I}}            
\newcommand{\tideal}{\widetilde{\ideal}}
\newcommand{\Sphere}{\o{S}}            

\newcommand{\Xx}{\mathfrak{X}}	
\newcommand{\Edge}{\mathscr{E}} 
\renewcommand{\L}{\o{L}} 
\newcommand{\MM}{\mathbf{M}} 
\newcommand{\tMM}{\widetilde{\MM}} 

\newcommand{\Quad}{\o{Quad}} 		 

\newcommand{\Kplus}{\xp{K}}               
\newcommand{\tiota}{\tilde{\iota}}

\newcommand{\Corner}{\mathscr{C}} 
\newcommand\Farey{\mathsf{F}}

\newcommand{\um}{u^-}
\newcommand{\up}{u^+}
\newcommand\orderpreserving{Corollary~5.5~} 
\newcommand\disjointnesscriterion{Theorem~6.2~} 

\title[Two-generator groups]{Proper affine deformations of two-generator
Fuchsian groups}
\author[Charette]{Virginie Charette}
    \address{{\it Charette:\/}
    D\'epartement de math\'ematiques\\ Universit\'e de Sherbrooke\\
Sherbrooke,  Qu\'ebec J1K 2R1  Canada}
    \email{v.charette@usherbrooke.ca}
\author[Drumm]{Todd A. Drumm}
    \address{{\it Drumm:\/}
    Department of Mathematics\\ Howard University\\
    Washington, DC 20059 USA }
    \email{tdrumm@howard.edu}
\author[Goldman]{William M. Goldman}
    \address{{\it Goldman:\/}
    Department of Mathematics\\ University of Maryland\\
    College Park, MD 20742 USA}
    \email{wmg@math.umd.edu}

\date{\today}

\thanks{Charette gratefully acknowledges partial support from the
Natural Sciences and Engineering Research Council of Canada.
Goldman gratefully acknowledges partial support from National
Science Foundation grant DMS070781. 
All three authors acknowledge support from U.S. National Science Foundation grants DMS 1107452,  DMS1107263, DMS 1261522 and ,in particular DMS1107367 
``Research Networks in the Mathematical Sciences: Geometric structures And Representation varieties" (the GEAR Network) to facilitate this research.}
\begin{abstract} 
A Margulis spacetime is a complete flat Lorentzian $3$-manifold $M$ with free
fundamental group. Associated to $M$ is a noncompact complete hyperbolic surface 
$\surf$ homotopy-equivalent to $M$. The purpose of this paper is to classify
Margulis spacetimes when $\surf$ is homeomorphic to a one-holed torus.
We show that every such $M$ decomposes into polyhedra bounded by
crooked planes, corresponding to an ideal triangulation of $\surf$.
This paper classifies and analyzes the structure of {\em crooked ideal triangles,\/}
which play the same role for Margulis spacetimes as ideal triangles play for hyperbolic surfaces.
This extends our previous work on affine deformations of
three-holed sphere and two-holed cross surfaces.
\end{abstract}
\maketitle
\tableofcontents
 \section{Introduction}\label{sec:intro}
A {\em Margulis spacetime\/} is a complete
flat Lorentzian $3$-manifold $M$ with free fundamental group.
Associated to $M$ is a complete noncompact hyperbolic surface
$\surf$ homotopy equivalent to $M$. The purpose of this
paper is to classify Margulis spacetimes whose associated
hyperbolic surface $\surf$ is homeomorphic to a one-holed torus 
$\Soo$. 
This is part of a larger program of classifying Margulis spacetime
whose fundamental group is isomorphic to the rank two free group $\Ft$.

The significance of $\Soo$ among surfaces with fundamental group $\Ft$
is that $\Soo$ is an ``topological avatar'' of $\Ft$ in the following sense:
Every automorphism $\Ft\longrightarrow\Ft$ 
arises from a homeomorphism $\surf\longrightarrow\surf$ and 
two homeomorphisms of $\surf$ are isotopic if and only if the corresponding
automorphisms of $\Ft$ are equal. Thus the topological properties of $\surf$ are
{\em completely equivalent\/}  to the algebraic properties of $\Ft$.
A purely topological reformulation is that every homotopy-equivalence $\Soo\longrightarrow\Soo$ is homotopic to a homeomorphism. 

This special signifance derives from J.\ Nielsen's theorem~\cite{nielsen1917} 
that every automorphism $\Ft\longrightarrow\Ft$ preserves the {\em
peripheral structure,\/} namely the subset
\[
 \{ P\in \Ft \mid  P
\text{~is conjugate to~} 
ABA^{-1}B^{-1}  \text{~or to ~}  BAB^{-1}A^{-1}  \}\]
where $\{A,B\}$ is a fixed set freely generating $\Ft$.
In particular, the {\em mapping class group\/} 
$\MCG$  %
is isomorphic to the {\em outer automorphism group\/} $\Out(\Ft)$.
(Compare the discussion of Nielsen's paper in \cite{MR680777}, Chapter II.2,
pp. 81--83.)

This paper builds upon our previous 
work~(\cite{MR2653729,charette2011finite})
which deals with the cases
when $\surf$ is homeomorphic to a $3$-holed sphere $\surf_{0,3}$ 
or a $2$-holed cross-surface (real projective plane) $C_{0,2}$,
respectively. 
Every surface whose fundamental group is isomorphic to $\Ft$ is
homeomorphic to either 
$\surf_{0,3}$, 
$C_{0,2}$, 
the one-holed Klein bottle $\Coo$ or the one-holed torus $\Soo$. 
However, unlike $\surf_{0,3}$ and $C_{0,2}$, 
the deformation spaces for $\Soo$ and $\Coo$ are considerably more complicated
than the cases already considered.
In particular, the deformation space is a triangle for $\surf_{0,3}$
and a quadrilateral for $C_{0,2}$, respectively.
For the one-holed torus $\Soo$ and the one-holed
Klein bottle $\Coo$, the deformation space has infinitely many sides,
In general the deformation space is a convex polygon which is not
strictly convex.

The complexity of the boundary in this case was apparent in 
an example in 
\cite{charette2006non}, 
demonstrating the necessity of infinitely
many conditions to characterize which affine 
deformations are proper. 
Other examples indicating the differences in the deformation spaces
between homotopy-equivalent but nonhomeomorphic surfaces 
are given in 
\cite{MR2295549}. 
Nonproper deformations for which the
Margulis invariants $\alpha(\gamma)$ are either all positive or all
negative are constructed in 
\cite{GLMM,danciger2013margulis}
answering negatively a question raised in
\cite{MR1796129,goldman2002margulis}.

In this paper we restrict to the the one-holed torus $\Soo$, 
deferring the discussion of the deformation space for the one-holed Klein bottle
$\Coo$ to a forthcoming paper.

%
Just as convex hyperbolic surfaces decompose into ideal triangles,
Margulis spacetimes decompose into analogous subsets,
called {\em crooked ideal triangles.\/} A key technique 
is to pass from ideal triangulations of a hyperbolic surface $\surf$ 
to a decomposition of the Margulis spacetime into crooked ideal
triangles. If $M^3$ is a Margulis spacetime whose linearization $\surf$
has $\chi(\surf) = -1$, then our main technique involves {\em crookedly realizing\/} 
an ideal triangulation on $\surf$, and showing that every such Margulis spacetime admits
such a crooked ideal triangulation.
This program was carried out  when $\surf\approx\Soo$ 
in \cite{MR2653729}  and when $\surf\approx \Coo$ in \cite{charette2011finite}.
The theory of crooked ideal triangles developed here is based on
our previous work on disjointness of crooked planes developed in
\cite{drumm1999geometry} and \cite{burelle2012crooked}.
The Structure Theorem for crooked ideal triangles (Theorem~\ref{thm:StructureTheorem})
implies that every nondegenerate crooked ideal triangle has a canonical decomposition
into a {\em minimal\/} crooked ideal triangle (one where the faces share a common vertex)
and {\em parallel crooked slabs,\/} namely regions bounded by a pair of parallel crooked planes.
 
Drumm's early work~\cite{drummThesis,drumm1992fundamental} 
(see also \cite{MR1796126}) 
develops a flat Lorentzian analog of Poincar\'e's fundamental polyhedron theorem: 
if a crooked polyhedron $P$ in Minkowski $3$-space $\E$ is equipped with 
isometric face-pairings generating a group $\Gamma$, 
then $\Gamma$ acts properly on $\E$ with fundamental domain $P$.
Using this idea, every noncompact complete hyperbolic surface with finitely generated fundamental group arises from a Margulis spacetime (\cite{drumm1993linear}). 
(The noncompactness of $\surf$ is due to Mess~\cite{mess2007lorentz},
and reproved by different methods in Goldman-Margulis~\cite{MR1796129}
and Labourie~\cite{MR1909247}.)

It was natural to conjecture that every Margulis spacetime is
{\em geometrically tame,\/ } that is, admits
a fundamental domain bounded by crooked planes. 
We finish the proof this conjecture when $\surf \approx \Soo$.

Since the results of this paper were announced, 
JeffreyDanciger, Fran\c cois Gu\'eritaud and Fanny Kassel 
proved~\cite{danciger2013margulis} this conjecture 
(when $\surf$ is convex cocompact)  using the arc complex of $\surf$.
Their approach relates closely to ours; one difference is that we use the {\em pants complex\/} of $\surf$ instead, and we use ideal triangulations rather than decompositions with geodesic segements orthogonal to the boundary.

One consequence of this conjecture is {\em topological tameness:\/} 
the Margulis spacetime $M^3$ is homeomorphic to a solid handlebody. 
When $\surf$ is convex cocompact, then proofs of this general fact
have recently been announced in \cite{choi2013topological,danciger2013geometry}.

The paper is organized as follows. Sections \S 2, \S 3 and \S 4 establish notation, terminology and background prerequisites concerning affine, Lorentzian and hyperbolic geometry. Section \S 5 gives background on the rank two free group $\Ft$ and its geometric avatar, the one-holed torus $\Soo$ and introduces the important
notion of a {\em superbasis\/} of $\Ft$.
A {\em superbasis} is an equivalence class of free bases,
which corresponds to a tile in the deformation space. 
Given a hyperbolic structure on $\Soo$, a superbasis corresponds to an ordered triple of simple unoriented closed geodesics, mutually intersecting  transversely in a single point. Such a structure determines an ideal triangulation. Superbases thus parametrize ideal triangulations of $\Soo$. 
Section \S 5 closes with a description of this theory in  the well known and elegant
theory of {\em Farey arithmetic.}

Section \S 6 contains background on affine deformations of Fuchsian groups,
and in particular reviews the Margulis invariant which gives coordinates on the space of affine deformations.
Just as a superbasis of $\Ft$ gives affine cooordinates of the $\SLt$-character variety of $\Ft$ (discussed in \S\ref{sec:CharVars}, based on \cite{MR2497777}),  
a superbasis identifies the vector space $\Ho$ of equivalence classes of affine deformations with $\R^3$. 

Section \S 7 develops the structure of crooked ideal triangles, and establishes Theorem~\ref{thm:StructureTheorem}.
This enables an explicit description of the moduli space of crooked ideal triangles as a $6$-dimensional 
{\em orthant\/} $\R_+^6$.
Section \S 8 applies this theory to compute the {\em tiles\/} of the deformation space
which correspond to ideal triangulations of $\surf$.
The final section \S 9 describes how every affine deformation of $\surf$ 
corresponds to a tile or an edge and admits a crooked ideal quadrilateral as a fundamental domain.

\section*{Acknowledgements}

We would like to thank Jean-Philippe Burelle, Jeffrey Danciger, Fran\c cois Gu\'eritaud,  Fanny Kassel, Greg Laun, Yair Minsky, and Ser Peow Tan for many interesting conversations. We are grateful to the hospitality of Princeton University,
Universit\'e de Sherbrooke, Howard University, Institut Henri Poincar\'e (Paris), 
Centro de Investigaci\'on en Matem\'aticas (Guanajuato), 
Korean Institute for Advanced Study, 
International Centre for Theoretical Sciences (India), 
and Centre de Recherches Math\'ematiques (Montr\'eal)
for various visits which enabled this collaboration.
Finally we thank the GEAR Research Network in Mathematical Sciences,
and other grants from the National Science Foundation and the National
Science and Engineering Council of Canada for their generous financial support during this collaboration.

\section{Notation and terminology}
\subsection{Elementary conventions}
By a {\em triple\/} (respectively {\em pair\/}) we shall
always mean an {\em ordered\/} triple or pair respectively,
unless otherwise stated.

The complement of a subset $X$ will sometimes
be denoted $X^c$.

Denote the ring of rational integers by $\Z$ and
the fields of rational, real and complex numbers by $\Q, \R$ and $\C$ respectively.
The multiplicative groups of nonzero (respectively positive) real numbers will 
be denoted $\R^*$ (respectively $\R_+$).
\subsection{Groups and transformations}
If $A$ is an element of a group, denote by $\langle A\rangle$ the cyclic subgroup
generated by $A$.

If $A, B$ are invertible transformations, denote the commutator by:
\[
[A,B] := ABA^{-1}B^{-1}
\]
and conjugation by:
\[
A \xmapsto{\Inn(B)} B A B^{-1} 
\]
For any group $G$, denote its 
identity element as 
$\Id$, 
the inverse operation as $g\longmapsto g^{-1}$
and its multiplication by $(a,b) \longmapsto ab$. Denote its
automorphism group by $\Aut(G)$.
Denote its normal subgroup of inner automorphisms by $\Inn(G)$
and the quotient $\Aut(G)/\Inn(G)$ by $\Out(G)$.

\begin{defn}\label{defn:Invert} 
An automorphism $\phi\in\Aut(G)$ {\em inverts\/} an element
$g\in G$ if and only if  $\phi(g) = g^{-1}$.\end{defn}

\subsection{Projectivization}\label{sec:Projectivization}
If $V$ is a vector space over a field $k$, and $\vv\in V$ is a 
vector, denote $k\vv$ the $1$-dimensional linear subspace
spanned by $\vv$. 
The {\em projective space\/} $\P(V)$ associated to $V$ is the
set of all lines 
\[
[\vv ] = \{ k\vv \, \vert \,  k\in\R \} \subset V, 
\]
where $\vv\neq 0$.
Defined the {\em projectivization mapping\/} by:
\begin{align*}
V\setminus\{0\} &\xrightarrow{\P} \P(V) \\
\vv &\longmapsto [\vv]
\end{align*}
Projectivization is the quotient map for action of $\R^*$ by homotheties.
Since $\R^*  = \R_+ \times \{\pm 1\}$, projectivization factors as the composition
of two quotient maps, one by the group of {\em positive scalings\/} $\R_+$ and
one by $\{\pm 1\}$. The {\em sphere of directions\/} $\Sphere(V)$, the set of rays
at $0$ in $V$,  is the quotient
space $\big(V\setminus\{0\}\big)/\R_+$.
The $\{\pm 1\}$-quotient map $\Sphere(V)\longrightarrow\P(V)$  is
a double covering with covering group $\{\pm 1\}$. 
Given a Euclidean structure on $V$, its unit sphere
is a covenient cross-section to the $\R_+$-quotient map 
$V\setminus\{0\} \longrightarrow \Sphere(V)$.
This sphere $\Sphere(V)$ double covers $\P(V)$
with covering group $\{\pm 1\}$.

\subsection{Surfaces and curves}

The notation for surfaces follows 
\cite{MR2497777}.
An $n$-holed orientable surface of genus $g$ is denoted 
$\surf_{g,n}$. Such a surface is represented 
as the complement of $n$ disjoint discs in a connected sum
$\surf_g$ of $g$
tori (each summand is $\surf_{1,0}$ %
and $\surf_0$ is a $2$-sphere). %
Similarly, $C_{g,n}$ is represented as the complement of $n$
disjoint discs in a connected sum of $g$ {\em cross-surfaces.\/}
(A {\em cross-surface\/} is the name proposed by John H.\ Conway for a topological space 
homeomorphic to 
the real projective plane.)

Let $\surf$ be a surface with boundary $\partial\surf$.
A closed simple curve on $\surf$ is {\em essential\/} if and only if
it is not homotopic to a point.
A simple closed curve is {\em peripheral\/} if and only if
it is homotopic to a curve in $\partial\surf$.
If $\partial\surf$ is connected, then a simple closed curve is
{\em nonseparating\/} if and only if it essential and nonperipheral.
(This simply means its complement is connected.)
By a {\em curve class\/} we shall mean an isotopy class of
essential nonperipheral simple closed curves. In particular the 
curves are not assumed to be oriented, and the isotopies
are not required to preserve a (so far unmentioned) basepoint.
Denote the set of curve classes by $\Sss$.
Denote the set of isotopy class of {\em oriented\/} essential nonperipheral simple closed curves by $\Sss^+$.

If $\alpha,\beta\in\Sss$, 
their {\em geometric intersection number\/} 
$i(\alpha,\beta)$ is the nonnegative
integer which is the minimum cardinality of intersections of
representative curves in $\alpha,\beta$ respectively.
If $\surf$ is given a hyperbolic structure, then $i(\alpha,\beta)$
equals the cardinality of the intersection of the unique closed geodesics
representing $\alpha$ and $\beta$ (unless $\alpha=\beta$ in which
case 
$i(\alpha,\beta) = 0$).

The {\em mapping class group\/} of $\surf$ is defined as the group of isotopy classes
of homeomorphisms of $\surf$:
\[
\MCG := \pi_0\big(\Homeo(\surf)\big). \]
Note that we do not restrict to orientation-preserving homeomorphisms,
and $\MCG$ is therefore defined even when $\surf$ is nonorientable.

\section{Affine and Lorentzian geometry}
This section develops background for affine geometry of 
{\em Minkowski space,\/} the affine space
corresponding to a Lorentzian vector space.

\subsection{Affine spaces and vector spaces}

An {\em affine space\/} is a set equipped by a simply transitive
action of a vector space. We denote the affine space by $\E$ 
and the vector space by $\V$. The elements of $\E$
are {\em points,\/} and the elements of $\V$ are {\em vectors.\/}
Vectors $\vv\in\E$  identify with transformations 
$\E\longrightarrow\E$, the {\em translations.}
If $p\in\E$ is a point and $\vv\in\V$ is a vector, then the effect
of translation by $\vv$ on $p$ is the point which we denote $p + \vv$. 
We denote the unique vector effecting the translation taking point $p$ to point $q$
by $q - p \in \V$; this is just the vector difference in standard
coordinates derived from a basis of $\V$. 

In this paper, all vector spaces are vector spaces over the field
$\R$ of real numbers.  Then $\E$ is a smooth manifold and 
the simply transitive action of $\V$ on $\E$ identifies $\V$ with each tangent
space $\T_p\E$ for each point $p\in\E$.
If $[a,b]\xrightarrow{\gamma}\E$ 
is a smooth curve in affine space, then for each parameter 
$t\in[a,b]$, the {\em velocity vector\/} 
\[
\gamma'(t)\in \T_{\gamma(t)}\E \longleftrightarrow \V
\]
is defined.
The differential-geometric context is the following:
$\E$ is a smooth $1$-connected
manifold with a geodesically complete flat torsionfree affine connection.
Conversely, every $1$-connected smooth manifold with geodesically complete flat torsionfree affine connection is isomorphic to such an 
affine space. 

When $\V$ is equipped with a positive definite inner product, 
the affine space $\E$ is {\em Euclidean space.\/}
 A Riemannian manifold $(M,g)$
is isometric to a Euclidean space $\E$ if and only if 
$M$ is $1$-connected, and the Riemannian metric $g$ is
geodesically complete and has zero curvature and torsion.

\subsection{Lorentzian isometries}
In this paper, $\V$ will be a {\em Lorentzian vector space,\/}
that is, a $3$-dimensional inner product space of signature $(2,1)$.
We denote the inner product of two vectors $\vv,\vu\in\V$ by 
$\vv\cdot\vu$. (We emphasize this notation is {\em not\/} the usual
Euclidean dot product.) 
When the $\V$ is {\em oriented,\/} 
the  associated {\em cross product\/}
\begin{align*}
\V \times \V &\xrightarrow{\times} \V \\
(\vu,\vv) &\longmapsto \vu\times \vv,
\end{align*}
is unambiguously defined by
\[ (\vu\times\vv) \cdot \vw = \Det(\vu,\vv,\vw). \]
(Here $\Det$ is the alternating trilinear form on $\V$, 
uniquely specified by having unit length with respect to the
induced inner product on $\Lambda^3(\V)^*$ and compatible
with  the chosen orientation on $\V$.)

The following standard identity will be useful
later:
\begin{equation}\label{eq:CrossProductIdentity}
(\vu_1\times\vv_1) \cdot (\vu_2\times\vv_2) =
-(\vu_1\cdot \vu_2) (\vv_1\cdot \vv_2)  
+(\vu_1\cdot \vv_2) (\vv_1\cdot \vu_2)  
\end{equation}

\subsection{Minkowski space}

Define {\em Minkowski space\/} as an affine space $\E$ whose underlying vector
space $\V$ has the structure of a Lorentzian vector space. 
Alternatively, $\E$ is a geodesically 
complete $1$-connected flat Lorentzian manifold.

The geodesics in $\E$ are just Euclidean straight
lines, parametrized at constant Euclidean
speed:
\[ t \longmapsto  p + t \vv \]
Depending on whether the (constant) velocity
vector $\vv$ is timelike, null, or spacelike,
we say the geodesic is
a {\em particle,\/}  {\em photon,\/} or 
{\em tachyon,} respectively.

The group of all isometries of Minkowski space $\E$ splits as a semidirect product
$\mathsf{O}(2,1) \ltimes \V$ where now $\V$ denotes the vector space of translations.
Its index-two subgroup of {\em orientation-preserving isometries\/} $\Iso$
is the semidirect product $\SOto \ltimes \V$, that is, the following sequence of groups
\[
0 \longrightarrow \V \longrightarrow \Iso \xrightarrow{\L} \SOto \longrightarrow 1
\]
is exact.

\subsection{Involutions in geodesics}
Suppose that $\vu\in\V$ is a vector which is not
null. Then 
\begin{equation}\label{eq:involution}
\vv \longmapsto  -\vv + 2\frac{\vv\cdot\vu}{\vu\cdot\vu} \vu
\end{equation}
defines an involution of $\V$ fixing $\vu$ having
$\vu^\perp$ as its $(-1)$-eigenspace. 
For timelike $\vu$, this involution
defines the symmetry of $\Ht$ 
in the corresponding point in $\Ht$; 
for spacelike $\vu$, this involution defines
reflection in the corresponding geodesic in $\Ht$. 

Let $T := p + \R\vt$ be a particle parallel to a
unit-timelike vector $\vt\in\V$ (that is,
$\vt\cdot\vt = -1$).
The corresponding  reflection is:
\begin{equation}\label{eq:ParticleInvolution}
p + \vv \stackrel{\tiota}\longmapsto
p - \vv - 2 (\vv\cdot\vt) \vt
\end{equation}

\section{Hyperbolic geometry}
\subsection{Points and orientation in $\Ht$}
Let $\V$ be an oriented Lorentzian vector space.
The {\em hyperbolic plane\/} $\Ht$ is defined as
the subset of the projective space $\P(\V)$ 
consisting of timelike lines (particles), that is, lines spanned
by vectors $\vt$ such that $\vt\cdot\vt < 0$. 

The set of all nonzero non-spacelike 
vectors (that is, vectors $\vv\neq 0$ such that
$\vv\cdot\vv \le 0$) falls into two components,
denoted $\Future$ and $\Past$. 
The choice of a component $\Future$
is equivalent to choosing a  {\em time-orientation\/} 
on $\V$ and is useful in discussing the
Riemannian metric and the orientation of $\Ht$.
Namely, we identify $\Ht$ with 
\[
\{ \vt\in\Future \mid \vt\cdot\vt = -1\}
\]
The Lorentzian structure on $\V$ induces a Riemannian metric on
$\Ht$ of constant curvature $-1$. 
Furthermore the orientation on $\V$ and the radial vector field
on $\V$ induce an orientation on $\Ht$, when $\Ht$ is
described as above. 
However, identifying $\Ht$ with the {\em past\/} vectors $\vt\in\Past$ 
satisfying $\vt\cdot\vt = -1$ gives $\Ht$ the opposite orientation.

The group $\SOto$ of orientation-preserving linear isometries of 
$\V$ isometrically on $\Ht$ as follows. 
Its identity component $\SOoto$ preserves $\Future$ and acts by 
orientation-preserving isometries of $\Ht$. (Here $\Ht$ inherits
an orientation from $\V$ and the radial vector field on $\V$.)
However, orientation-preserving linear isometries of $\V$  not
in $\SOoto$ interchange $\Future$ and $\Past$. 

Here is how we interpret $\Ht$ in terms of the affine Lorentzian geometry of 
Minkowski space $\E$. The translations of $\E$ act on particles 
(timelike geodesics) in $\vt\subset\E$. Points in $\Ht$ then correspond
to {\em translational equivalence classes\/} of particles in $\E$.
We denote the equivalence class of $\vt$ as $[\vt]\in\Ht$. 
Similarly, points of $\partial\Ht$ correspond to translational equivalence
classes of null geodesics in $\V$ and geodesics in $\Ht$ correspond
to translational equivalence classes of tachyons (spacelike geodesics) in $\E$.
See \cite{burelle2012crooked} for further details.

\subsection{Geodesics and halfplanes in $\Ht$}
Geodesics in $\Ht$ correspond to spacelike $1$-dimensional
linear subspaces of $\V$.  {\em Oriented geodesics\/} 
in $\Ht$ correspond to unit-spacelike vectors.
The projectivization $\P(\vv^\perp)$ of the orthogonal
complement $\vv^\perp$ of a spacelike vector $\vv\in\V$ 
meets $\P(\V_-) \approx \Ht$ in a geodesic.
A geodesic in $\Ht$ separates $\Ht$ into two  {\em halfplanes.}

{\em Time-orientations\/} of $\V$ can be used to conveniently 
parametrize halfplanes in $\Ht$, as follows.
Choose a connected component
$\Future$ of
\[ \big\{ \vv\in\V\setminus\{0\} \mid \vv\cdot\vv < 0 \big\}. \]
Then $\P(\Future) = \Ht$ and restriction of the quotient map
\[ \V\setminus\{0\} \longrightarrow \P(\V) \]
identifies $\Ht$ with
\[ \big\{ \vv\in\Future \mid \vv\cdot\vv = -1 \big\}. \]
Under this identification, a spacelike vector $\vs$ determines a
{\em halfplane:}
\[ \HH(\vs) := \big\{ \vv\in\Ht \mid \vv\cdot\vs \ge 0 \big\} \]
bounded by the geodesic $\P(\vs^\perp)$.
The complement of the geodesic $\P(\vs^\perp)$
consists of the interiors of the two halfplanes
$\HH(\vs)$ and $\HH(-\vs)$.

The time-orientation on $\V$
(together with the orientation on $\V$)
induces an orientation on $\Ht$.
In this way an {\em oriented geodesic\/} $l\subset\Ht$ determines a halfplane 
bounded by $l$.

\subsection{Ideal polygons}\label{sec:IdealPolygon}

Define an {\em ideal polygon\/} $\Pi$ in $\Ht$ as a closed submanifold-with-boun\-dary
in $\Ht$ whose boundary is a disjoint union of complete 
mutually asymptotic  geodesics, which we call its {\em sides.\/} 
The complement of an ideal polygon is a disjoint union of (open) halfplanes,
and each side bounds a halfplane. 
If $\Pi$ is finite-sided, then the union of $\Pi$ with the set of endpoints of its
sides is compact.

An {\em ideal triangle\/} is an ideal polygon with three sides and an 
{\em ideal quadrilateral \/} is an ideal polygon with four sides.
A {\em pointed ideal polygon\/}  is an ideal polygon with a choice of a point on
each one of its sides.

\subsection{Isometries of $\Ht$}
For hyperbolic $X$,
denote by $\xp{X},\xm{X}$ respectively its attracting and repelling
fixed points.  When $X$ is parabolic, then denote by
$\xp{X} = \xm{X}$ its single fixed point.
If $X,Y\in\SOoto$ and $X$ is hyperbolic or parabolic, 
then 
\begin{align*}
Y\big(\xp{X}\big)\ &=\  \xp{(\Inn(Y)X)}\ = \xp{(YXY^{-1})}, \\
Y\big(\xm{X}\big)\ &=\ \xm{(\Inn(Y)X)}\ = \xm{(YXY^{-1})}.
\end{align*}

\subsection{Discrete subgroups of $\SOto$}
\label{sec:FuchsianGroups}
A discrete group $\Gamma_0$ of isometries of $\Ht$ acts
properly on $\Ht$. 
Furthermore a discrete group $\Gamma_0$ acts freely
if and only if $\Gamma_0$ is torsionfree.
The quotient $\surf = \Ht/\Gamma_0$ is a complete hyperbolic
surface with fundamental group isomorphic to $\Gamma_0$.

Let $x\in\Ht$. Then the closure of the orbit $\Gamma_0(x)$ in
$\Ht \cup \partial\Ht$ is independent of $x$, and is called the 
{\em limit set\/} of $\Gamma_0$, and denoted $\Lambda$. 
If $\Gamma_0$ is {\em nonelementary,\/} that is, not virtually
abelian, then $\Lambda$ is the unique minimal closed 
$\Gamma_0$-invariant subset of $\partial\Ht$. Geodesic rays
in $\Ht$ approaching a point of $\Lambda$ project to nonwandering
(or weakly recurrent) geodesic rays on the quotient. 

\begin{figure}[b]
\centerline{\includegraphics[scale=1.]{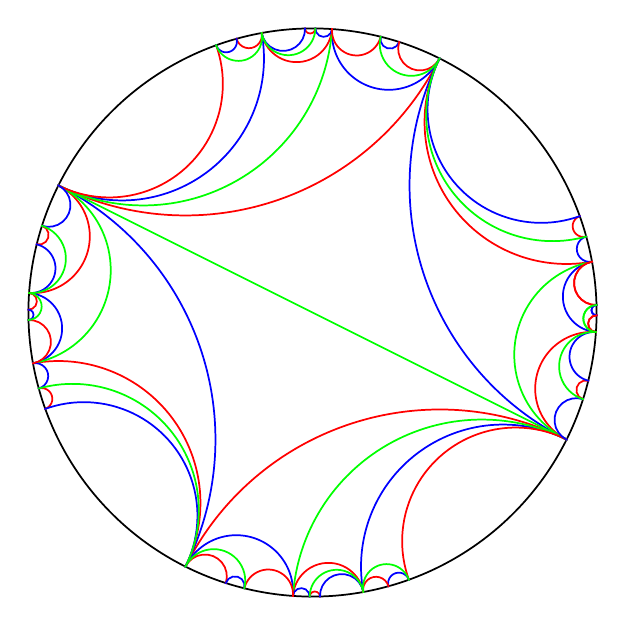}}
\caption{Nielsen convex region for a square 1-punctured torus tesselated by ideal triangles}
\label{fig:NielsenRegion}
\end{figure}

The {\em Nielsen convex region\/} of $\Gamma_0$ is the 
{\em convex hull\/} $\mathsf{Hull}(\Lambda)$ of $\Lambda$ in $\Ht$. 
Its quotient, the {\em convex core \/}  $\mathsf{Core}(\surf)$,
is the unique minimal convex subsurface carrying $\pi_1(\surf)$:
{\em Any geodesically convex subsurface $\surf'\subset\surf$
such that the inclusion $\surf'\hookrightarrow\surf$ is a 
homotopy equivalence must contain $\mathsf{Core}(\surf)$.}
Figure~\ref{fig:NielsenRegion} depicts a Nielsen convex region
in $\Ht$,
tiled by ideal triangles.

\subsection{Affine deformations of Fuchsian groups}
Consider a Fuchsian group
$\Gamma_0\subset\SOto$ which is free of rank two.
Denote the embedding onto $\Gamma_0$ by:
\[ \Ft 
\stackrel{\rho_0} \hookrightarrow\SOto \] 
%
An {\em affine deformation\/} of $\Gamma_0$ 
is a lift $\rho$ of $\rho_0$, namely a homomorphism such that
$\rho_0= \L\circ\rho$,  that is, such that the diagram
\begin{equation*}
\xymatrix{
& \Iso \ar[d]^{\L}  \\
\Ft \ar@{-->}[ur]^{\rho} 
\ar@{^{(}->}[r]^{\rho_0~} & \SOto}
\end{equation*}
commutes, 
Translational conjugacy classes of affine deformations form the vector space
$\Ho$, where $\V$ is the $\Ft$-module defined by the representation
$\Ft\hookrightarrow\SOto$. 
Explicitly, $\rho= (\rho_0,u)$, where  
$\Ft\xrightarrow{u}\V$ is the translational part:
\[
x \xmapsto{\quad\rho(A)\quad } \rho_0(A) x + u(A),
\]
then $u\in Z^1(\Ft,\V)$ is a cocycle taking values in the 
$\Ft$-module defined by $\rho_0$. 

In this paper, we will work in a more geometric context,
where Margulis spacetimes are {\em affine deformations\/} 
of the hyperbolic surface $\surf$. 
Equivalently, $\rho_0$ corresponds to the marked hyperbolic surface 
\[ \surf_0 := \Ht/\Gamma_0 \] with an isomorphism 
$\pi_1(\surf_0) \cong \Ft$.
This isomorphism induces an isomorphism
\[ \ho \cong\Ho \]
and the cohomology space $\ho$ is the space of all 
{\em affine deformations of the hyperbolic surface $\surf_0$.}

\subsubsection{Scalings}\label{sec:scalings}

Since the group $\R^*$ of scalar multiiplications of $\V$ centralize the linear part
$\rho_0$, it acts on the deformation space $\Ho$, again by scalar multiplications.
To exploit this extra symmetry, we quotient $\Ho$ by $\R^*$.
Thus the proper affine deformation space may be conveniently considered
as a subset of the projective space $\P\Ho$.

However, as will be discussed in \S\ref{sec:MargulisInvariant}, the Margulis invariant
is the key to understanding proper affine actions, but it depends on a choice of 
orientation. 
To maintain orientation, we only quotient by the group $\R_+$ of positive
scalings to obtain the {\em sphere of directions\/} $\Sphere\Ho$,
as in \S\ref{sec:Projectivization}. 
Two affine deformations are conjugate by the scaling $\lambda\in\R_+$  
if and only if their linear parts agree and their translational parts differ by 
$\lambda$, and hence determine the same point in $\Sphere\Ho$. 


\subsubsection{The Margulis invariant}\label{sec:MargulisInvariant}
 
Conjugacy classes of an individual (nonelliptic) element in an affine deformation of a group are classified by
the \emph{Margulis invariant}. Suppose $X\in\SOoto$ is hyperbolic or parabolic. Define a fixed vector $X^0$ such that:
\begin{itemize}
 \item (direction) for any timelike vector $\vu$, 
 \[ \Det( \vu, X(\vu) , X^0 ) >0, \]
and
\item  (length) if $X$ is hyperbolic then $X^0\cdot X^0 = 1$ or if $X$ is parabolic the length of $X^0$ 
is chosen at random and fixed.
\end{itemize}
If $A = \big( X, \vv)$ where $X = \L(A)$, 
define the {\em Margulis invariant\/} of $A$:
\[ \alpha(A) = X^0 \cdot \vv  . \]
%
Since the fixed vector $(X^{-1})^0$ of $X^{-1}$ equals $-X^0$,
the Margulis invariant is invariant under inversion:
\[ \alpha(A^{-1}) = \alpha(A). \]
The Margulis invariant is preserved under translational conjugacy~\cite{MR2138887}, and  any $\vv$ 
as above is translationally conjugate to a unique vector in the direction of $X^0$. For hyperbolic $X$ with no fixed points, it
is interpreted as the {\em signed\/} Lorentzian length
of the unique closed geodesic in  $\E/ \langle A \rangle$.
The following elementary properties are proved in 
\cite{abels2001properly,MR2106472,
MR2138887,drumm2001isospectrality,MR1796129,
margulis1983free,margulis1987complete}:

\begin{itemize}
\item $\alpha (A) = 0$ if and only if $A$
fixes a point;\label{hpfact:neq0}
\item For any $\eta\in\Iso$,
$\alpha(BAB^{-1}) \,=\, \alpha(A);$
\item For any $n\in\Z$,
$\alpha(A^n) \,=\, \vert n\vert  \alpha(A).$
\end{itemize}
\noindent 
The following {\em Opposite Sign Lemma\/} is less elementary.
It was first proved in \cite{margulis1983free,margulis1987complete};
see Abels~\cite{abels2001properly} for further discussion and 
Goldman-Labourie-Margulis~\cite{MR2600870} for a different proof.

\begin{lemma}[Margulis]\label{lem:OppositeSign}
Let $A,B\in \Iso$ have non-elliptic linear part and suppose 
$\alpha(A)$ and $\alpha(B)$ have opposite signs. 
Then $\langle A,B\rangle$ does not act properly on $\E$.
\end{lemma}

This paper concerns  proper actions, and Lemma~\ref{lem:OppositeSign} implies that
translational conjugacy classes of proper affine deformations form two subsets
of $\Ho$, namely those with positive Margulis invariants, and those with negative Margulis invariants. These two subsets are interchanged by the antipodal automorphism $-\Id$,
which reverses orientation. 
For this reason we restrict attention to proper affine deformation with 
whose Margulis invariants are {\em positive.\/} 
Since this subset of $\Ho$ is $\R_+$-invariant, and positive scalings simply conjugate
the actions by dilations (and preserve $\L\circ \rho$) as in \S\ref{sec:scalings},
it becomes convenient to consider the proper deformation space as a subset of the
sphere of directions $\Sphere\Ho$. 
(See \S   \ref{sec:Projectivization}.) 

\section{The rank two free group }\label{sec:F2}

This purely algebraic section discusses the free group $\Ft$.
Conjugacy classes of primitive elements of $\Ft$ correspond to primitive elements
in its abelianization $\Z^2$. 
Furthermore, the outer automorphism group $\Ou$
of $\Ft$ is naturally isomorphic to the automorphism group $\Aut(\Z^2) = \GLtZ$.
The geometry of primitive elements is conveniently defined by {\em Farey fractions,\/}
which correspond to points in the rational projective line $\Qpo$.
Following Conway~\cite{MR1478672},
we introduce {\em superbases,\/}  which are the vertices of a natural trivalent tree 
$\T$ upon which the automorphism group $\GLtZ$ acts.

The natural trichotomy of primitives in $\Ft$ corresponds to a map or projective
lines $\Qpo \longrightarrow \P^1(\Z/2)$ defined by reduction modulo $2$. 
This gives a natural ordering to the set of edges of $\T$ incident to a given vertex.

Superbases are used to define affine coordinates on the $\o{SL}(2)$-character variety of $\Ft$ (an old theorem of Vogt) in \S\ref{sec:CharVars}.
Its infinitesimal version describes the vector space of translational conjugacy classes of affine deformations of an irreducible $\o{SL}(2)$-representation of $\Ft$
in terms of the Margulis invariants applied to a superbasis 
(\S\ref{sec:AlphaCoordinates}).

Finally, in \S\ref{sec:CoxExtFtRep} the Coxeter extension of $\Ft$ is described in terms of the {\em elliptic involution\/} of $\Ft$ defined in \S\ref{sec:EllInv}.
This is the algebraic manifestation of ``hyper-ellipticity''
of the one-holed torus and plays a fundamental role in our parametrization
fo Margulis spacetimes. 

\subsection{Bases and primitive elements of $\Ft$}\label{sec:bases}
A {\em basis\/} of $\Ft$ is an
ordered pair $(A,B)\in \Ft\times\Ft$ such that
$A$ and $B$ generate $\Ft$.
(Necessarily $A$ and $B$ {\em freely\/} generate $\Ft$.)
Denote the set of bases of $\Ft$ by $\BB$.
A {\em basic triple\/} is an ordered triple 
\[(A,B,C)\in\Ft\times\Ft\times\Ft \]
with $ ABC=\Id$, such that $(A,B)\in\BB$.
Clearly bases extend uniquely to basic triples by defining $C := \bm\am$. 
Conversely any pair inside a basic triple is a basis.

The automorphism group $\Aut(\Ft)$ acts simply transitively
on $\BB$: if $(A_0,B_0)\in\BB$ and $(A,B)\in\Ft\times\Ft$, then the mapping
\begin{align*}
A_0 &\longmapsto A \\
B_0 &\longmapsto B
\end{align*}
extends uniquely to an endomorphism $\Ft\longrightarrow\Ft$.
This endomorphism is an automorphism if and only if $(A,B)\in\BB.$

An element $A\in\Ft$ is {\em primitive\/} if and only if
$A$ extends to a basis. 
Let $\Prim(\Ft)\subset\Ft$ denote the set of primitives in $\Ft$.
If $A\in\Prim(\Ft)$, then the set of elements $B\in\Ft$ such that
$(A,B)\in\BB$ is the double coset $\langle A\rangle B_0\langle A\rangle$
where $B_0$ is an arbitrary element of this set.

\subsection{Automorphisms of $\Ft$.}
The group $\Inn(\Ft)\cong\Ft$ of inner automorphisms
acts  on the set $\Prim(\Ft)$.
We find a double extension $\Inne(\Ft)$ of 
$\Inn(\Ft)$ such that the quotient 
$\Prim(\Ft)/\Inne(\Ft)$ identifies with the
set $\Sss$ of free isotopy classes of {\em unoriented\/} noseparating 
simple closed curves on $\surf$.

\subsubsection{Abelianization}\label{sec:Abelianization}
Automorphisms $\phi\in\Aut(\Ft)$ induce automorphisms of the
abelianization 
$\langle A\rangle\oplus \langle B\rangle  \cong \Z^2,$
thus defining a homomorphism
\[ \Aut(\Ft) \xrightarrow{\Abelianize} \GLtZ \]
Inner automorphisms abelianize trivially,
and indeed $\Inn(\Ft) = \Ker(\Abelianize),$
and $\Abelianize$ induces an isomorphism
\[ \Ou \xrightarrow{\cong} \GLtZ.\]
The center of $\GLtZ$ equals the two-element subgroup $\{ \pm 1 \}$, and
\[
\Inne(\Ft) := \Abelianize^{-1}(\{ \pm 1 \}) \]
is a normal subgroup of $\Aut(\Ft)$ containing $\Inn(\Ft)$ with index two.

\subsubsection{Elliptic involutions}\label{sec:EllInv}
If $(A,B)$ is a basis, then {\em the elliptic involution $\e_{A,B}$
associated to $(A,B)$\/} is the automorphism 
which {\em inverts\/} $A$ and $B$ 
(in the sense of Definition~\ref{defn:Invert}):
\begin{align*}
\Ft & \xrightarrow{\e_{A,B}} \Ft \\
A &\longmapsto A^{-1} \\
B &\longmapsto B^{-1} 
\end{align*}
Although  $\e_{A,B}$ inverts $A$ and $B$, it does not invert the third element
$C$ of the corresponding basic triple.

\begin{defn*} An {\em elliptic involution\/} of $\Ft$ is any automorphism of $\Ft$
which is $\Aut(\Ft)$-conjugate to $\e_{A,B}$. \end{defn*}
\begin{lemma}
Let $\phi\in\Aut(\Ft)$ be an automorphism. 
The following are equivalent:
\begin{enumerate}
\item $\phi$ is an elliptic involution.
\item There exists a basis $(A',B')$  such
that $\phi$ inverts $A'$ and $B'$.
\item $\phi^2 = \Id$ and $\Abelianize(\phi)=-1$.
\end{enumerate}
\end{lemma} 



\subsubsection{Homomorphism to the modular group}
The subgroup $\Inne(\Ft)\subset\Aut(\Ft)$ 
defined in \S\ref{sec:Abelianization} is 
generated by $\e_{A,B}$ and $\Inn(\Ft)$. 
Furthermore
\begin{equation*}
\Inne(\Ft) \; \cong \;  \Fte 
\end{equation*}
and $\Inne(\Ft)/\Inn(\Ft)$ equals the center of $\Ou$ and
under the isomorphism
\begin{equation*}
\Aut(\Ft)/\Inn(\Ft) \;\cong\; \Ou \;\cong\; \GLtZ
\end{equation*}
the class $\e := [\e_{A,B}]$ maps to the central generator $-\Id$ of 
$\GLtZ$.
In particular 
\begin{equation*}
\Aut(\Ft)/\Inne(\Ft) \;\cong\; \Mod.
\end{equation*}

\subsection{The Farey tree}
\label{sec:Farey}
The set $\Prim(\Ft)/\Inne \cong \Qpo$ enjoys a rich combinatorial structure
which we encode in a trivalent tree $\T$ with a $\GLtZ$-action. 
In \S\ref{sec:1ht} this will be related to topological objects on $\surf$.

\subsubsection{Primitives and $\Qpo$}\label{sec:Primitives}

The simplest approach involves the abelianization $\Z^2$ of $\Ft$. 
The set $\Prim(\Z^2)$ consists of pairs $(p,q)\in\Z^2$ such thatz:
\begin{itemize}
\item At most one of $p,q$ is zero;
\item If both $p,q$ are nonzero, the $p$ and $q$ are coprime;
\item If $p = 0$, then $q = \pm 1$;
\item If $q = 0$, then $p = \pm 1$.
\end{itemize}
Such a primitive $(p,q)$ generates a one-dimensional linear subspace of $\Q^2$,
and therefore corresponds to a point in the {\em rational projective line\/} $\Qpo$.
The resulting map 
\[ \Prim(\Z^2) \longrightarrow \Qpo \]
is the quotient mapping by the action of $\{\pm\Id\}$ on $\Z^2$ by scalar multiplication.

An affine patch in $\Qpo$ consists of those points with homogeneous
coordinates $q\neq 0$; the corresponding point in inhomogeneous
(affine) coordinates is $\frac{p}{q}\in\Q$. 
The complement consists of the unique ideal point 
$\infty$ corresponding to $\frac10$ so $\Qpo = \Q \cup \{\infty\}$.

If $g\in\Ft$ is primitive, then its abelianization $\Abelianize(g)\in\Z^2$ is primitive,
and the corresponding mapping
\[
\Prim(\Ft) \xrightarrow{\Abelianize} \Prim(\Z^2) \]
is surjective. Furthermore it is equivariant with respect to the quotient homomorphism
$\Inne \longrightarrow \{\pm\Id \} $ with kernel $\Inn(\Ft)$ and defines an isomorphism
\begin{equation}\label{eq:PrimF2Qpo}
\Prim(\Ft)/\Inne \xrightarrow{\cong}  \Qpo. \end{equation}
This follows from the special property that two primitives with the same abelianization
are conjugate.

\subsubsection{Farey neighbors and superbases}
Two primitives $(p_1,q_1),(p_2,q_2)\in\Prim(\Z^2)$ generate $\Z^2$ if and and only
if $p_1 q_2 - p_2 q_1  = \pm 1$. 
In that case they form a basis of $\Z^2$ and we say that the corresponding points
\[ x_i := \frac{p_i}{q_i}\in\Qpo \] 
are {\em Farey neighbors.\/}
A triple $(x_1,x_2,x_3)\in \big(\Qpo\big)^3$ is a {\em Farey triple\/} if and only if
any $x_i$ and $x_j$ are Farey neighbors for $i\neq j$.

This closely relates to the notion of a {\em superbasis,\/} introduced by
Conway~\cite{MR1478672}: 
\[
(\alpha,\beta,\gamma) \in \Z^2 \times \Z^2 \times \Z^2 \]
is a {\em superbasis\/} if $\alpha + \beta + \gamma = 0$ and $(\alpha,\beta)$ is a basis of $\Z^2$. In that case $(\beta,\gamma)$ and $(\gamma,\alpha)$ are also bases of
$\Z^2$. 

Given  a pair of Farey neighbors $x_1,x_2$ as above, there are 
{\em exactly two extensions \/} to Farey triples $(x_1,x_2,x_3), (x_1,x_2,x_3')$,
given by {\em Farey addition\/} and {\em subtraction}:
\[
x_3 := \frac{p_1 + p_2}{q_1 + q_2},\qquad
x_3' := \frac{p_1 - p_2}{q_1 - q_2}. \]

The complex whose vertices corresponding to Farey triples of $\Qpo$ 
and edges corresponding to pairs of Farey neighbors, 
is a tree $\T$:
If $v$ is a vertex and $e$ is an edge, then $v$ is an endpoint of $e$ if
the pair of primitives correspond to $e$ lies in the triple of primitives corresponding
to $v$.

\subsubsection{Reduction modulo $2$ and three types of primitives}
Every superbasis admits a canonical ordering, which can easily be described using
Farey arithmetic and the canonical homomorphism $\Z \longrightarrow \Z/2$. 

The projective line $\P^1(\Z/2)$ contains three elements
\[ \frac10, \frac01, \frac11 \]
 corresponding to $\infty, 0, 1 (\mod~2)$ respectively.
Furthermore it is the ``algebraic avatar'' of  $3$-element sets:
Every permutation of $\P^1(\Z/2)$ is a projective transformation.
In other words,
\[
\o{GL}(2,\Z/2) 
\cong \SymThree. \]
For example, the generating 2-cycles for $\SymThree$ are represented
by matrices
\[
(12) \longleftrightarrow \bmatrix 0 & 1 \\ 1 & 0 \endbmatrix,\qquad
(23) \longleftrightarrow \bmatrix 1 & -1 \\  0 & -1 \endbmatrix \equiv
\bmatrix 1 &  1\\  0 & 1  \endbmatrix (\mod~2) \]
where the first choice of matrices defines a splitting $\SymThree \longrightarrow
\GLtZ$. 
Reduction modulo $2$ defines a homomorphism $\Z^2 \longrightarrow (\Z/2)^2$,
which  maps $\Prim(\Z^2)$ to the three nonzero elements of $(\Z/2)^2$. 
The resulting map 
\[
\Sss \cong \Qpo \longrightarrow \P^1(\Z/2)  \cong \{\infty, 0, 1\; (\mod~2)\}\]
partitions primitives in $\Ft$) into three classes,
depending on whether the reduced fraction $p/q$ corresponding to a coprime pair $(p,q)\in\Z^2$ 
respectively satisfies:
\begin{align*}
\frac{p}{q}     =\;  \frac{\text{odd}}{\text{even}}&\longleftrightarrow\frac10 = \infty, \\
\frac{p}{q}     =\; \frac{\text{even}}{\text{odd}}&\longleftrightarrow\frac01 = 0, \\
\frac{p}{q}     =\; \frac{\text{odd}}{\text{odd} }&\longleftrightarrow\frac11 = 1. 
\end{align*} 
If $\{x_1,x_2,x_3\}$ is an unordered Farey triple, then one may apply a permutation in $\SymThree$ so that 
\begin{align*}
x_1 & \equiv 0~(\mod~2), \\
x_2 & \equiv \infty~(\mod~2), \\
x_3 & \equiv 1~(\mod~2).
\end{align*}
We call such an ordered Farey triple {\em canonically ordered.\/}

\subsection{ Superbases of $\Ft$}\label{subsec:Super}
Superbases arise in several closely related contexts dealing with $\Ft$;
in particular, superbases describe affine coordinate systems on 
deformation spaces of $\SLt$-representations of $\Ft$. 
They will also parametrize the affine deformations, and also the tiles in the 
affine deformation space which correspond to ideal triangulations of a hyperbolic
surface $\surf\approx\Soo$.
After defining superbases in \S\ref{sec:BasicTriplesSuperbases}, 
we explain their combinatorial structure, 
which appears in the proper affine deformation space. 
Specifically, the set of {\em unordered\/} superbases defines the set of vertices of
a binary trivalent tree, with edges corresponding to the relation of
{\em neighboring superbases.\/} 

The previous discussion of superbases in $\Z^2$ motivates the present discussion.
Farey triples in $\Qpo$ precisely correspond to superbases in $\Ft$,
our principal object of study. 

The key point is the following.
If $(A,B)\in\BB$ is a basis, then 
$\big(\Abelianize(A),\Abelianize(B)\big)$ is a basis of the free abelian group $\Z^2$
and the corresponding points are Farey neighbors in $\Qpo$.
Conversely, a pair of Farey neighbors in $\Qpo$ determines a basis of $\Ft$,
unique up to the action of $\Inne$.

\subsubsection{Basic triples and superbases}\label{sec:BasicTriplesSuperbases}
An automorphism $g\in\Au$ acts on a basic triple $(A,B,C)$ 
(as defined in \S\ref{sec:bases}), as follows: 
\begin{equation}\label{eq:AuAction}
(A,B,C) \xmapsto{g} \big(g(A), g(B), g(C) \big) \end{equation}
\begin{defn*}
An {\em ordered superbasis\/} of $\Ft$ an  $\Inne$-orbit of basic  triples.
\end{defn*}
\noindent
Denote the set of ordered superbases by $\BBB$.
The  $\Au$-action \eqref{eq:AuAction} induces an $\Ou$-action on $\BBB$.
\begin{prop}
$\Mod 
\cong\Ou$
acts simply transitively on $\BBB$. \end{prop}
\begin{proof}
The automorphism group $\Au$ acts simply transitively on the set $\BB$
of bases of $\Ft$, which identifies with the set of basic triples. 
Therefore $\Ou\cong\GLtZ$  acts transitively on 

\[ \BBB = \BB/\Inne. \] 
The kernel of this action is the center of $\GLtZ$, 
which is generated by the image  $- \Id\in  \GLtZ$ of (any) elliptic involution.
Thus $\Mod$ acts effectively, and hence simply transitively, on $\BBB$.
\end{proof}

\subsubsection{Ordered and unordered superbases}
Although the natural action of the symmetric group $\SymThree$
on $\Ft\times\Ft\times\Ft$ does not preserve the set of basic triples,
$\SymThree$ does act on superbases:
First, observe that $\SymThree$ acts on basic triples as follows.
Define a $\SymThree$-action on basic triples in terms of the 
$2$-cycles $(12)$ and $(13)$  generating $\SymThree$:
\begin{align}\label{eq:PermutationsOf BasicTriples}
(A,B,C) &\xmapsto{~P_{(12)}} (B^{-1},A^{-1},C^{-1}) \\
(A,B,C) &\xmapsto{~P_{(23)}}  (A^{-1},C^{-1},B^{-1}) \notag
\end{align}
This $\SymThree$-action commutes with the $\Inne$-action on basic triples
and induces a $\SymThree$-action on the set $\BBB$ of ordered superbases.
Define an {\em unordered superbasis\/} as a $\SymThree$-orbit in $\BBB$ of ordered superbases.

Using the mod $2$ reduction, every unordered superbasis admits a {\em canonical ordering.\/}
Namely if $\{x,y,z\}$ is an unordered superbasis, then 
exactly one of $x,y,z$ reduces mod 2 to $[0:1]\longleftrightarrow\infty$,
exactly one of $x,y,z$ reduces mod 2 to $[1:0]\longleftrightarrow 0$, and 
exactly one of $x,y,z$ reduces mod 2 to $[1:1]\longleftrightarrow1$.
Applying a permutation, we may assume that $x \equiv \infty (\mod~2)$,
$y \equiv 0 (\mod~2)$, and $z \equiv 1 (\mod~2)$. 
Call such a superbasis {\em canonically ordered.\/}

\subsubsection{Neighboring superbases}

Two different (unordered) superbases are {\em neighbors\/} 
if they share two elements.  
For example, a neighbor of the superbasis $(X,Y,Z)$ is the superbasis $(X,Y,W)$. If $(A,B,C)$ is an ordered superbasis such that $\psi$ maps $A,B,C$ to $X,Y,Z$, respectivley, then the ordered 
superbasis $(A, B^{-1}, B A^{-1})$ corresponds to the different superbasis $(X,Y, W)$ since $W$ and $Z$ are not conjugates.  
Each superbasis has three superbasis neighbors coming from this construction on 
any pair of the superbasis. 

A pair of neighboring (unordered) superbases corresponds to an {\em edge\/}
in the tree whose vertices are unordered superbases. 
Equivalently such an edge corresponds to an $\Inne$-orbit of bases of $\Ft$.
They also correspond to $\Inne$-orbits of elliptic involutions.

These edges also correspond to the {\em edges\/} in the tiling of proper affine deformation space described later in this paper.

\subsection{Superbases and coordinates on moduli spaces}
Superbases describe global coordinates on moduli spaces of $\SLt$-representations
of $\Ft$ and affine deformations. This remarkable fact was apparently first published
by Vogt~ \cite{vogt1889invariants} and plays a fundamental role in our calculations.

\subsubsection{Affine coordinates on the $\SLt$-character variety}\label{sec:CharVars}
The {\em representation variety\/} $\Hom(\Ft,\SLtC)$ is an affine variety
upon which $\Inn(\SLtC)$ acts by composition.
Let $\Xx(\Ft,\SLtC)$ denote the $\SLtC$-character variety of $\Ft$,
that is, the {\em categorical quotient\/} of $\Hom(\Ft,\SLtC)$ by the 
(conjugation) action of $\Inn\big(\SLtC\big)$. 
That is,
every $\Inn\big(\SLtC\big)$-invariant regular function
on  $\Hom(\Ft,\SLtC)$ arises from a regular function
on  $\Xx(\Ft,\SLtC)$.
A superbasis defines an isomorphism:
\begin{align*}
\Xx(\Ft,\SLtC) & \stackrel{\chi}\longrightarrow \C^3 \\
[\rho] &\longmapsto \bmatrix 
\tr\big(\rho(A)\big) \\
\tr\big(\rho(B)\big) \\
\tr\big(\rho(C)\big) 
\endbmatrix
\end{align*}
(Vogt's paper~\cite{vogt1889invariants} 
seems to be the first published proof of this result.
See~\cite{MR2497777} for a modern discussion of these results
and 
\cite{MR680777}, p. 41,83 for historical discussion of
\cite{vogt1889invariants}. )

Denoting the functions
$\tr\big(\rho(A)\big)$,
$\tr\big(\rho(B)\big)$,
$\tr\big(\rho(C)\big)$
by $\tr_A,\tr_B,\tr_C$ respectively,
the statement above means that
every $\Inn\big(\SLtC\big)$-invariant regular function
on  
\[
\Hom(\Ft,\SLtC) \cong \SLtC \times \SLtC
\]
equals a polynomial 
$f\big(\tr_A,\tr_B,\tr_C\big)$
where $f$ is a polynomial in $3$ variables. 
Restriction of $\chi$
yields  {\em trace coordinates\/}
on the Fricke spaces of surfaces $\surf$ with
$\pi_1(\surf) \cong \Ft$.

Restricting the mapping
\[
\Hom(\Ft,\SLtC) \longrightarrow \Xx(\Ft,\SLtC) 
\]
to the subset of {\em irreducible representations\/} 
gives the quotient mapping for the action of $\Inn\big(\SLtC)$ on 
$\Hom(\Ft,\SLtC)$.
In particular two irreducible representations have the same character
if and only if they are $\SLtC$-conjugate. 
(Compare \cite{MR2497777}.)

\subsubsection{Margulis invariants as coordinates on affine deformations}\label{sec:AlphaCoordinates}

If $\V$ is the flat $\SOto$-vector bundle
determined by the holonomy of a hyperbolic structure on $\surf$ as above, then a superbasis defines an isomorphism
\begin{align*}
H^1(\Ft, \V) & \longrightarrow \R^3 \\
[u] & \longmapsto \bmatrix
\alpha_{[u]}(A) \\
\alpha_{[u]}(B) \\
\alpha_{[u]}(C) \endbmatrix
\end{align*}
where $\alpha_{[u]}$ denotes the Margulis invariant
(as described in \S\ref{sec:MargulisInvariant}) of 
the affine deformation defined by the cohomology
class $[u]$ of a cocycle $u\in Z^1(\Ft,\V)$. 
For the proof, see Lemma~6.1 of 
\cite{MR2653729}.
This may be regarded as the differentiated version of Vogt's theorem.
Since $\SLtR$-lifts of holonomy representations of hyperbolic structures on
non elementary surfaces are irreducible,
elliptic involutions act identically on the tangent bundle $T\FF(\surf)$  to Fricke
space. Therefore every affine deformation $\Ft \longrightarrow\Iso$ admits
a unique Coxeter extension $\Fte\longrightarrow\Iso$ as well.
Compare 
\cite{charette2003affine}.

Let $\Sphere\Ho$ denote the quotient of $\Ho\setminus\{0\}$ under 
positive scalings. By Lemma~\ref{lem:OppositeSign}, together with
scale-invariance, it suffices to consider the quotient space 
of nonzero affine deformations by $\R_+$.


\subsection{Coxeter extensions of  $\Ft$-representations}\label{sec:CoxExtFtRep}
Representations of $\Ft$ in $\SLtC$ closely relate to 
$\PSLtC$-representations of the free product
$\Z/2\star\Z/2\star\Z/2$ of three groups of order two, 
which we  denote by 
\[
\Fte := \langle \iota_1,\iota_2,\iota_0 \mid
\iota_1^2 = \iota_2^2 =\iota_0^2 = 1   \rangle.
\]
Specifically, a generic representation of $\Ft= \langle A, B \rangle$ in $\PSLtC$ extends
to a representation of $\Fte$, where $\Ft\hookrightarrow\Fte$
is defined as follows:
\begin{align*}
A &\longmapsto  \iota_2\iota_0 \\
B &\longmapsto  \iota_0\iota_1 \\
C &\longmapsto  \iota_1\iota_2 
\end{align*}
The notation $\Fte$ emphasizes that
$\Fte$ is the double (split) extension of $\Ft= \langle A,B\rangle$ 
defined by the elliptic involution (which inverts the generators $A$ and $B$):
\[ 
\Fte = \langle A, B, \iota_0 \mid  \iota_0^2 = 1,\ 
\iota_0 A \iota_0 = A^{-1},\ 
\iota_0 B \iota_0 = B^{-1} \rangle
\]
An irreducible representation
$\Ft\xrightarrow{\rho_0}\SLtC$ admits a unique 
{\em Coxeter extension,\/} 
that is, a lift $\Fte\xrightarrow{\rho_0^\e}\SLtC$
such that 
\[ \xymatrix{
\Fte \ar@{-->}[rd]^{\phi^\e} \\
\Ft \ar@{_{(}->}[u]
\ar@{->}[r]^{\rho_0^\e}
&\SLtC }\]
commutes. (Compare~\cite{MR2497777}.)

Here is a simple proof. 
The trace function
\begin{align*}
\SLtC &\longrightarrow  \C \\
X &\longmapsto \tr(X)
\end{align*}
is invariant under both conjugation and inversion.
Since the elliptic involution $\e_{A,B}$ inverts $A$ and $B$ and conjugates $C$,
its acts trivially on $\X$. 
Thus, for any irreducible representation $\rho_0$ as above.
$\rho_0$ and $\rho_0\circ \e_{A,B}$ are conjugate by some $g\in\SLtC$.
Since $\rho_0$ is irreducible, Schur's lemma implies that $g$ is unique.
The extension of $\rho_0$ to $\rho_0^\e$ is defined by mapping $\e_{A,B}$ to
$g$.

Here is another notable consequence.
Since all elliptic involutions are conjugate,
the action of $\Ou\cong\GLtZ$ on $\X$ factors through its quotient
$ \Ou/\langle[\e_{A,B}]\rangle\ \cong\ \Mod$.

%

\section{The   one-holed torus}\label{sec:1ht}
This short section is purely topological.
Topological properties of the one-holed torus $\Soo$ correspond exactly to 
algebraic properties of its fundamental group $\Ft$.
Therefore the algebraic constructions from the \S\ref{sec:F2}
have topological interpretations, presented here:
Primitives correspond to isotopy classes of oriented nonseparating curves.
The {\em mapping class group\/} $\MCG$ is isomorphic to the automorphism group
$\Ou \cong \GLtZ$.
The tree $\T$ corresponds to the {\em pants complex\/}
of $\surf$.

\subsection{The fundamental group}\label{sec:FundGp}
Let $\surf$ be a one-holed torus.
For any basepoint $x_0\in\surf$, 
the fundamental group $\pi_1(\surf,x_0)$ is isomorphic to $\Ft$. 
Changing the basepoint amonts to apply inner automorphisms of $\Ft$,
and free homotopy classes of based curves correspond to conjugacy classes
in $\Ft$.

Choose an oriented nonseparating curve $A$ on $\surf$ starting at $x_0$.
Choose another oriented simple closed curve $B$ intersecting $A$ once at $x_0$,
such that the intersection is transverse.
Then the complement $\surf \setminus A$ compactifies 
to a three-holed sphere $\surf'$ with a quotient map
\begin{equation*}
\surf' \xrightarrow{f} \surf
\end{equation*}
identifying two boundary components
$\partial_+(\surf')$ and $\partial_-(\surf')$  of $\surf'$ to the original curve $A$.
Choose respective points $x_\pm \in \partial_\pm(\surf')$ such
that $f(x_+) = f(x_-)$. 
The preimage $f^{-1}(B)$ is an arc in $\surf'$ joining $x_-$ to $x_+$
which identifies to a simple closed curve $B$ which intersects $A$
in a basepoint
\begin{equation*}
x_0\;:=\; f(x_+)\; =\; f(x_-).
\end{equation*}
Their based homotopy classes define a basis for 
the fundamental group $\pi_1(\surf,x_0)$, 
and, in particular, are primitives.
Furthermore the boundary of $\surf$ is the image of a based loop
corresponding to the commutator $[A,B]\in \pi_1(\surf,x_0)$.

Enlarging the presentation by including two new generators
$C,K$ is useful. 
Namely, let $C$  correspond to a simple loop such that $ABC = 1$ 
in $\pi_1(\surf,x_0)$ to obtain a basic triple.
Let $K$ correspond to a simple loop freely homotopic to $\partial\surf$.
Then the fundamental group admits the following presentation:
\begin{equation}\label{eq:present}
\pi_1(\surf,x_0)\cong \Ft = 
\langle A,B,C,K~\mid~ C=(AB)^{-1},~K=[A,B]\rangle
\end{equation}
with $(A,B,C)$ the corresponding basic triple.

Nielsen~\cite{MR680777} proved that, for any basis
$(A',B')$ of $\pi_1(\surf,x_0)$, the commutator $[A',B']$ is conjugate to either
$[A,B]$ or $[A,B]^{-1}$. Topologically, this means that every self-homotopy equivalence
of $\surf$ is homotopic to a homeomorphism. 
Moreover, the mapping class group $\o{Mod}(\surf)$ of $\surf$ identifies
with $\Ou\;\cong\; \GLtZ$. 

\subsection{Primitives and curve classes}
Let $\Sss^+$ denote the set of free isotopy classes of {\em oriented\/} nonseparating curves.
If $A'$ is a separating curve, applying the classification of surfaces to the 
complement $\surf\setminus A'$ implies that
$A' \simeq \phi(A)$  for some self-homeomorphism $\phi$ of $\surf$.
Representing $A'$ by a based loop determines an element of $\Ft$,
which is also primitive, and well-defined up to conjugacy, that is, the action of
$\Inn(\Ft)$.
The resulting map
\[ \Sss^+ \longrightarrow \Prim/\Inn(\Ft) \]
is an isomorphism, and 
Thus  $\MCG$ acts transitively on $\Sss^+$.

Now we consider {\em unoriented\/} nonseparating curves.
Changing the orientation of a based loop corresponds to inverting
its based homotopy class in $\pi_1(\surf,x_0)$. 
Since the elliptic involution $\e_{A,B}$ inverts $A$, the quotient
$\Prim/\Inne(\Ft)$ corresponds to the set
$\Sss$ of free isotopy classes of {\em unoriented\/} nonseparating curves.

Abelianization $\Abelianize$ topologically corresponds to the map
\[
\pi_1(\surf,x_0)\cong\Ft \longrightarrow H_1(\surf,\Z)\cong \Z^2. \]
Another remarkable property of $\Soo$ is that two oriented nonseparating curves
on $\Soo$ are homologous if and only if they are freely isotopic.
This implies the fact (already mentioned in \S\ref{sec:Primitives}) that
the natural map $\Prim(\Ft) \longrightarrow\Prim(\Z^2)$ induces an isomorphism
\[ 
\Prim(\Ft)/\Inn(\Ft) \cong \Prim(\Z^2). \]
Hence, taking homology classes determines an isomorphism
\[
\Sss^+ \longrightarrow \Prim(\Ft)/\Inn(\Ft) \cong \Prim(\Z^2) \]
Combining this with the discussion in \S\ref{sec:Primitives}, gives an isomorphism:
\[ \Sss \xrightarrow{~\cong~} \Prim(\Ft)/\Inne(\Ft) \cong 
\Prim(\Z^2)/\{\pm\Id\} \cong \Qpo \]

\subsection{Topological superbases}\label{sec:TopSuperB}
Geometric intersection number defines a map
\[ \Sss \times \Sss \xrightarrow{i} \Z_{\ge 0} \]
which, in the case of the one-holed torus $\surf$ corresponds to just the
absolute value of the determinant. 
Namely if $A_i\in\Sss$ corresponds to equivalence classes $\pm(p_i,q_i)$ 
of coprime pairs for $i=1,2$, then
\[ i(A_1,A_2) = \vert p_1q_2 - p_2q_1 \vert. \]
In particular Farey neighbors correspond to curve classes which have
geometric intersection number $1$.

Similarly, superbases admit a topological interpretation.
Since each two-element subset of an unordered basic triple is
a basis of $\Ft$, each element of a superbasis is a primitive element of $\Ft$.
As in \S\ref{sec:FundGp}, primitive elements correspond to noseparating simple 
closed curves on $\surf$. Thus each element of a superbasis corresponds
to a curve class, that is, an isotopy class of an unoriented nonseparating simple closed curve. 
A triple of curve classes
\[ (a,b,c) \in \Sss \times \Sss \times \Sss \]
defines a superbasis if and only if the corresponding geometric intersection numbers 
satisfy \[i(a,b) = i(b,c) = i(c,a) = 1.\]

\section{Hyperbolic one-holed tori}

This section treats the geometry of a hyperbolic surface $\surf \approx \Soo$.
That is, $\surf = \Ht/\Gamma_0$ where $\Gamma_0\subset\SOto$ is a Fuchsian group
freely generated by $\rho_0(A)$ and $\rho_0(B)$.
The basis $(A,B)$ of $\pi_1(\surf,x_0) \cong \Ft$ corresponds to a 
{\em marking\/} of $\surf$.
Primitives in $\Ft$ correspond to nonseparating closed geodesics on $\surf$. 
The elliptic involution $\e_{A,B}$ defines an isometric involution on $\surf$, 
with quotient orbifold a disc with three order two branch points.
This corresponds to extending $\Gamma_0$ to the Coxeter group $\Gamma_0^\e$
generated by inversions $\rho_0^\e(\iot{0}), \rho_0^\e(\iot{1}), \rho_0^\e(\iot{2})$
in three points in $\Ht$.

We show how the basis $(A,B)$ of 
$\pi_1(\surf) \cong \Ft$ determines an ideal triangulation of $\surf$.
This procedure constructs a fundamental domain for $\Gamma_0$ acting on 
its Nielsen convex region, which is an ideal quadrilateral $\QQ$.
Furthermore this ideal quadrilateral is the union of two ideal triangles,
having a common side which is a diagonal of $\QQ$.

Other bases of $\pi_1(\surf,x_0) \cong \Ft$ determine other ideal triangulations,
related by {\em elementary moves\/} on a the generators of the Coxeter group 
$\Gamma_0^\e$. 
The basic operation, a {\em flip,\/} involves replacing the diagonal of $\QQ$
by the other diagonal.

\subsection{The fundamental quadrilateral}\label{sec:FundQuad}

Let $\rho_0$ be a Fuchsian representation for hyperbolic surface $\surf\approx\Soo$
and let $(A,B)$ be a fixed basis for $\pi_1(\surf,x_0)\cong\Ft$.  
As in \S\ref{sec:CoxExtFtRep}, the representation $\rho_0$ extends to a representation
of the double extension $\Fte \supset \Ft$, where $\Fte$ is the Coxeter group
freely generated by involutions $\iot{0},\iot{1},\iot{2}$ and 
$A = \iot{2}\iot{0},B = \iot{0}\iot{1},C = \iot{1}\iot{2}$ is a basic triple.
Let $\t{i}$ denote the future-pointing unit timelike vector fixed by $\iot{i}$ for
$i=0,1,2$, so that $\t{i} = \Fix(\iot{i}) \in \Ht$.

Since the involutions preserve orientation, 
$\rho_0(A),\rho_0(B),\rho_0(C)$ preserve orientation and are transvections,
and any pair of these transvections form a basis.
Furthermore, $\iot{0}$ inverts $\rho_0(A)$, 
so $\t{0}\in\Axis\big(\rho_0(A)\big)$, and similarly for $B$. 
Similar arguments for $\iot{1}$ and $\iot{2}$ imply:
\begin{align*} 
[\t{0}] & = \Axis\big(\rho_0(A)\big) \cap \Axis\big(\rho_0(B)\big) \\
[\t{1}] & = \Axis\big(\rho_0(B)\big) \cap \Axis\big(\rho_0(C)\big) \\
[\t{2}] & = \Axis\big(\rho_0(C)\big) \cap \Axis\big(\rho_0(A)\big).
\end{align*}

Furthermore, the Coxeter extension provides a canonical square root for the boundary generator $K$:
\[ K\  :=\  [A,B]\  =\  (\iot{2}\iot{1}\iot{0})^2 \]
Since $\rho_0$ is the holonomy of a complete hyperbolic structure on $\Soo$,
the boundary holonomy $\rho_0(K)$ is either parabolic or hyperbolic. 
Since the $\rho_0(\iot{i})$ preserve orientation, 
the square root $\rho_0(\iot{2}\iot{1}\iot{0})$ is either parabolic or hyperbolic as well. 

Now we construct a pointed ideal triangle whose edges contain the respective fixed points
$\t{i} = \Fix(\iot{i})$ for $i=0,1,2$: 
Since $\iot{0}\iot{1}\iot{2}$ is not elliptic, choose a point 
$[\vn]\in\partial\Ht$  fixed by $\iot{0}\iot{1}\iot{2}$.
(Here $\vn$ denotes a nonzero null vector in $\V$ and $[\vn]$
its projective equivalence class.) In the case when 
$\iot{0}\iot{1}\iot{2}$ is hyperbolic there are two choices for $[\vn]$, and
 if $\iot{0}\iot{1}\iot{2}$ is parabolic the equivalence class $[\vn ]$ is 
uniquely determined.

Note that $\iot{0}$ interchanges $[\vn]$ and $\iot{0}[\vn]$.
Similarly, $\iot{1}$ interchanges $\iot{0}[\vn]$ and $\iot{1}\iot{0}[\vn] = \iot{2}[\vn]$,
and $\iot{2}$ interchanges $\iot{2}[\vn]$ and $[\vn] = \iot{2}[\vn]$.
Define:
\begin{align*}
\s{0} &:= \vn \times \iot{0}\vn; \\
\s{1} &:= \iot{0}\vn \times \iot{1}\iot{0}\vn; \\
\s{2} &:= \iot{1}\iot{0}\vn \times \vn
\end{align*}
Then $\s{0}^\perp, \s{1}^\perp, \s{2}^\perp$ define the three sides of an ideal triangle
$\Delta$. The respective endpoints of these sides are:
\[ \{[\vn],\iot{0}[\vn]\},\qquad \{\iot{0}[\vn],\iot{2}[\vn]\},  \qquad \{\iot{2}[\vn], [\vn]\}  \]
and these sides contain $\t{0}, \t{1}, \t{2}$ respectively.
We call the sequence
\[
\big([\vn],\iot{2}[\vn],\ \iot{1}\iot{2}[\vn] ,\ \iot{0}\iot{1}\iot{2}[\vn]\  = \ [\vn]\big)
\]
a {\em fixed point cycle.\/}
(See Figure~\ref{fig:FundamentalQuadrilateral}.) 

To obtain the fundamental quadrilateral,
{\em pivot\/} $\Delta$ around $\t{0}$, that is, apply the
the inversion $\iot{0}$ to $\Delta$, obtaining an ideal triangle
$\iot{0}\Delta$. 
The vertices of $\iot{0}\Delta$
form the fixed point cycle 
$(\iot{0}[\vn], \iot{0}\iot{2}[\vn],[\vn])$.
The union
\[
\QQ := \Delta \cup \iot{0}\Delta
\]
is an ideal quadrilateral
which forms a {\em fundamental domain\/} for $\Gamma_0$, which can viewed in 
Figure~\ref{fig:Identifications}.

The complement $\QQ^c$ is the disjoint union of four
halfplanes $\Hh_A^-, \Hh_B^-, \Hh_A^+, \Hh_B^+$ whose
boundaries are the respective sides of $\QQ$:
\begin{itemize}
\item
The side $\partial\Hh_A^-$  contains $\iot{0}\t{2}$ 
and has endpoints $\iot{0}\iot{2}\vn$ and $\iot{0}\vn$;
\item
The side $\partial\Hh_B^-$ contains $\t{1}$ 
and has endpoints $\iot{0}\vn$ and $\iot{2}\vn$;
\item
The side $\partial\Hh_A^+$ contains $\t{2}$ 
and has endpoints $\iot{2}\vn$ and $\vn$;
\item
The side $\partial\Hh_B^+$  contains $\iot{0}\t{1}$ 
and has endpoints $\vn$ and $\iot{0}\iot{2}\vn$.
\end{itemize}

\noindent
Then $A = \iot{2}\iot{0}$ maps $\partial\Hh_A^-$ to $\partial\Hh_A^+$ and
$B = \iot{0}\iot{1}$ maps $\partial\Hh_B^-$ to $\partial\Hh_B^+$.
Furthermore
\begin{align*}
A( \Hh_A^-) & = (\Hh_A^+)^c,\\
B( \Hh_B^-) & = (\Hh_B^+)^c,
\end{align*}
and the attracting/repelling fixed points satisfy:
\begin{align*}
\rho_0(A)^{\pm} & \in \Hh_{\rho_0(A)^{\pm}} \\ 
\rho_0(B)^{\pm} & \in \Hh_{\rho_0(B)^{\pm}}
\end{align*}
By Poincar\'e's theorem on fundamental polygons, 
such an ideal triangle $\Delta$ defines a fundamental domain for 
the Coxeter group 
\[
\Gamma_0^\e := \langle\iot{0},\iot{1},\iot{2}\rangle
\]
generated by the involutions
$\iot{0},\ \iot{1},\ \iot{2}$
acting on the Nielsen convex region
of $\Gamma_0^\e$. 
(Compare Figure~\ref{fig:NielsenRegion}.)

\begin{figure}
\centerline{\includegraphics[scale=.7]
{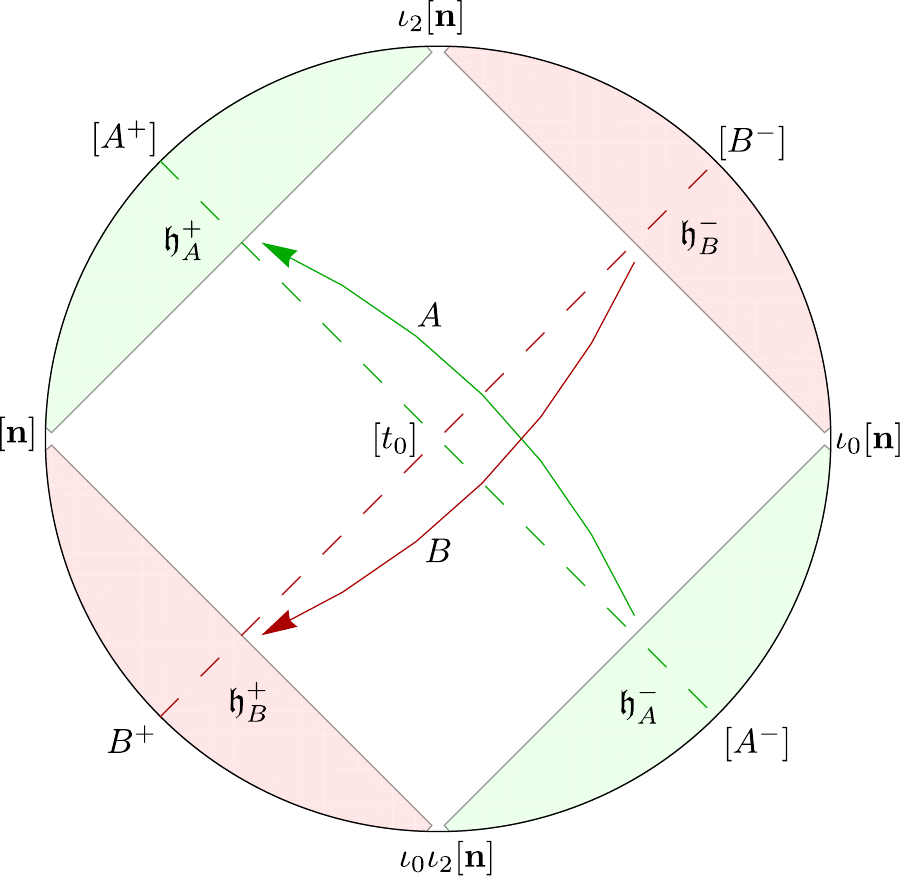}}
\caption{Identifications for a one-holed torus group}
\label{fig:Identifications}
\end{figure}
\begin{figure}
\centerline{\includegraphics[scale=0.7]{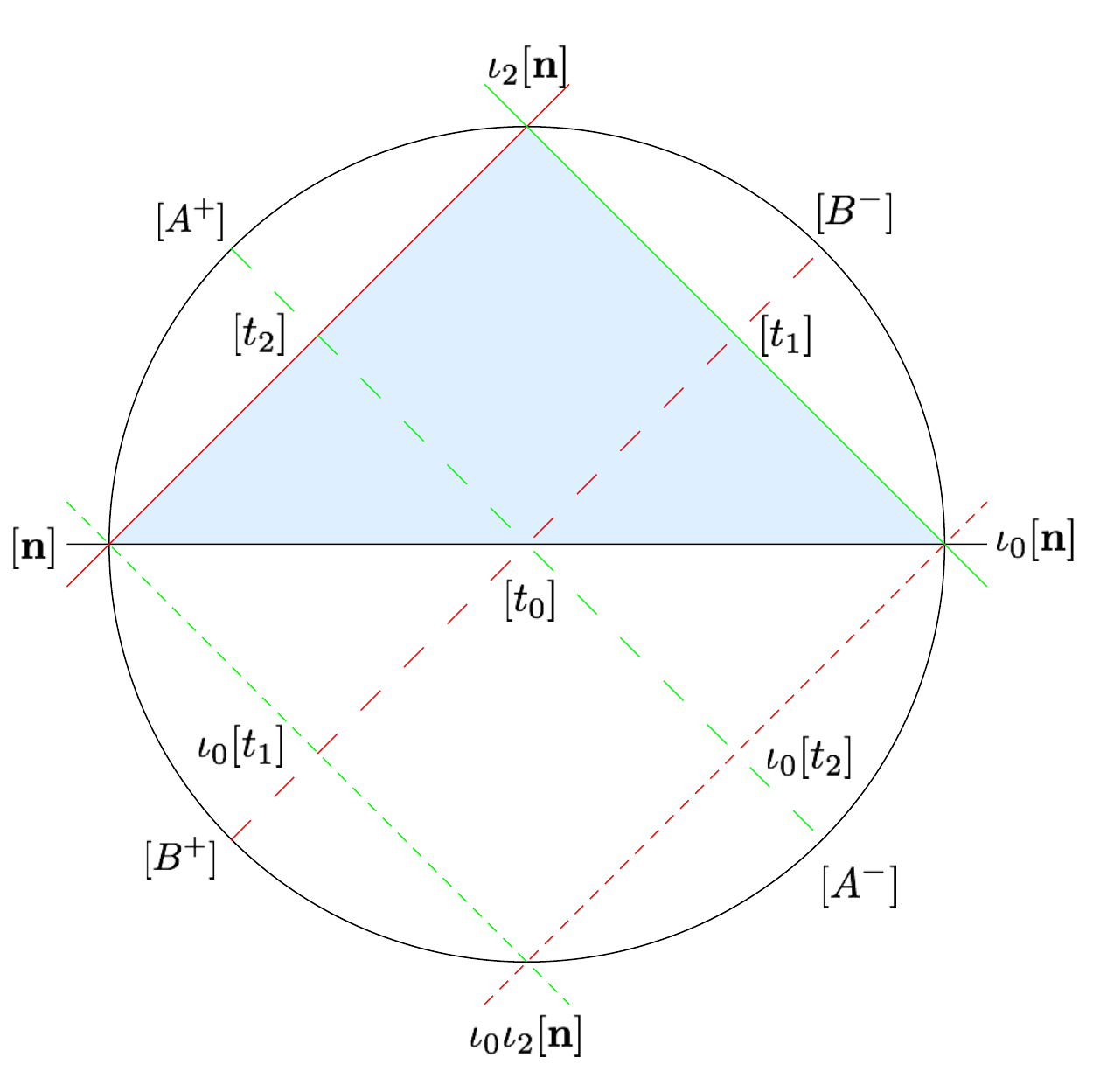}}
\caption{An ideal triangle with three points defining a Coxeter group}
\label{fig:FundamentalQuadrilateral}
\end{figure}

\begin{figure}
\centerline{\includegraphics[scale=0.7]{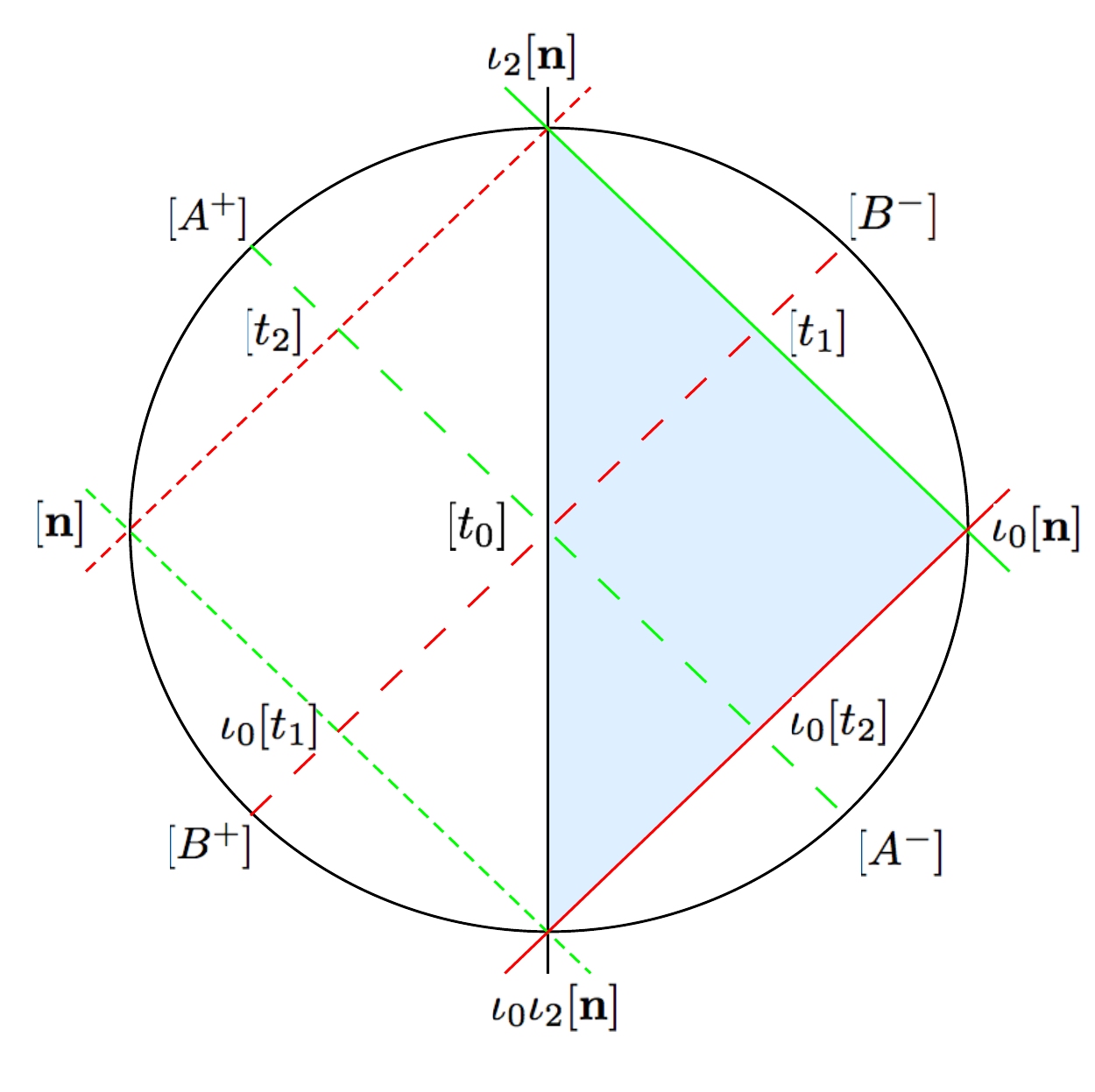}}
\caption{The flipped image of a triangle in an ideal quadrilateral}
\label{fig:FlippedTriangle}
\end{figure}
Geometrically, $\e_{A,B}$ corresponds to the involution of $\surf$  fixing the intersection 
point the simple closed geodesics corresponding to $A$ and $B$ respectively. 
The quotient orbifold with fundamental group $\Gamma_0^\e$ has fundamental domain
the pointed ideal triangle $\Delta$. The generators $\rho_0(\iot{i})$ identify the $i$-th
edge of $\partial\Delta$ with itself, and the quotient is a one-holed disc with three branch
points of order two, corresponding to the orbits of $\t{i}$ respectively.

The $\Gamma_0^\e$-orbit of $\Delta$ defines a decomposition of $\Ht$ into pointed
ideal triangles, that is, and {\em ideal triangulation.\/} 
This ideal triangulation is $\Gamma_0$-invariant and induces an ideal triangulation
of $\surf = \Ht/\Gamma_0$.

\subsection{Flipping the ideal triangulation}\label{sec:FlippingIdealTriangulation}
The vertices of $\QQ$ are (in clockwise order):
\[ [\vn],\quad \iot{2}[\vn],\quad \iot{0}[\vn],\quad \iot{0}\iot{2}[\vn] \]
The diagonal with endpoints $\{ [\vn],\iot{0}[\vn]\}$ divides $\QQ$ into two pointed ideal triangles:
\begin{alignat*}{2}
\Delta: &\text{~vertices~} \{[\vn],\iot{2}[\vn],\iot{0}[\vn]\}, 
&&\text{~points~} \{[\t{0}],[\t{1}],[\t{2}]\};\\
\iot{0}\Delta:&\text{~vertices~} \{[\vn],
\iot{0}\iot{2}[\vn],\iot{0}[\vn]\},
&&\text{~points~} \{[\t{0}],[\iot{0}\t{1}],[\t{2}]\}.
\end{alignat*}
The other diagonal of $\QQ$ has endpoints $\{\iot{2}[\vn],\iot{0}\iot{2}[\vn]\}$ 
and divides $\QQ$ into pointed ideal triangles:
\begin{alignat*}{2}
\Delta': &\text{~vertices~} \{\iot{2}[\vn],\iot{0}[\vn],\iot{0}\iot{2}\vn\}, 
&&\text{~points~} \{[\t{0}],[\t{1}],\iot{0}[\t{2}]\};\\
\iot{0}\Delta':&\text{~vertices~} \{\iot{0}\iot{2}[\vn],[\vn],\iot{2}[\vn]\},
&&\text{~points~} \{[\t{0}],[\iot{0}\t{1}],[\t{2}]\}.
\end{alignat*}
(Compare Figures~\ref{fig:FundamentalQuadrilateral},\ref{fig:FlippedTriangle}
where the flipping corresponds to a clockwise $90^o$ rotation.)

The new pointed ideal triangles determine a new ideal triangulation, as well as a new
superbasis. Namely the involution generators transform as follows:
\[ (\iot{0},\iot{1},\iot{2}) \longmapsto (\iot{0},\iot{0}\iot{2}\iot{0},\iot{1})  \]
and the superbasis transforms:
\[ (A,B,C) \longmapsto 
(B^{-1}, A, A^{-1}B =  A^{-2} C^{-1})  \]

\section{Crooked ideal triangles}\label{sec:CrookedIdealTriangles}

This section introduces {\em crooked ideal triangles,\/}
from which we build fundamental polyhedra 
for proper affine deformations.
This is completely analogous to the the decomposition of
hyperbolic surfaces into ideal triangles in $\Ht$. 
Crooked ideal triangles are polyhedra bounded by crooked
planes, which play a crucial role in this paper.
For the definition and theory of crooked planes, 
see \cite{burelle2012crooked,
MR2653729,charette2011finite,
MR1796126,
drumm1999geometry}.
Crooked planes were first introduced in \cite{drummThesis}.

Crooked ideal triangles {\em linearize\/} to ideal triangles in $\Ht$.
To every ideal triangle $\Delta\subset\Ht$ and a point $p\in\E$,
is a unique {\em minimal crooked ideal triangle\/} bounded by
crooked planes with vertex $p$ which linearizes to $\Delta$.
We prove a structure theorem that every crooked ideal triangle
is a union of a unique minimal crooked ideal triangle and three
{\em parallel crooked slabs.\/}

\subsection{Ideal triangles and crooked ideal triangles}\label{sec:IdealTriangle}


\begin{defn*}
A {\em crooked ideal triangle\/} 
is a region with boundary
$\Tt$  
bounded by three crooked planes
and exterior to three disjoint open halfspaces
\[ \H{0},\H{1},\H{2} \]
whose directing vectors
$\s{0}, \s{1}, \s{2}\in \V$ correspond to the sides of an ideal
triangle $\Delta\subset\Ht$.
\end{defn*}
\noindent
Call $\Delta$ the {\em linearization\/} of $\Tt$,
and denote it
\[ \Delta := \L(\Tt). \]
Conversely, call $\Tt$ an {\em affine deformation\/}  of $\Delta$.
The three crooked planes $\partial\H{i}$ are  the {\em faces\/} of
$\Tt$. 
An affine deformation of an ideal triangle $\Delta$ is 
described up to translational conjugacy by a triple of vectors
(called the {\em vertex triple\/}
and denoted ${(p_0,p_1,p_2)}$) parallel to the respective stems.

\begin{defn*}
A crooked ideal triangle 
\[
\Tt = \big( \H{0}\cup\H{1}\cup\H{2}\big)^c
\]
is {\em nondegenerate\/} 
if the closures of the three crooked halfspaces $\H{0},\H{1},\H{2}$ are pairwise disjoint.
\end{defn*}
\noindent
A nondegenerate crooked ideal triangle is a $3$-dimensional 
PL manifold-with-boundary.
Its boundary consists of three crooked planes.
Furthermore these three crooked planes are disjoint and
pairwise asymptotic. Conversely, three disjoint pairwise asymptotic
crooked planes bound a unique nondegenerate crooked ideal triangle.

To each nondegenerate crooked ideal triangle, we associate a degenerate 
crooked ideal triangle where the faces share a common vertex:

\begin{defn*}
A crooked ideal triangle $\Tt$ is {\em minimal\/} if its three faces
share a common vertex.
\end{defn*}
\noindent
Minimal crooked ideal triangles clearly are determined by the pairs $(O,\Delta)$
where $O\in\E$ is the common vertex and the ideal triangle $\Delta$ is its linearization:
Consider an ideal triangle $\Delta\subset\Ht$ and a point $O\in\E$.
Define the {\em corresponding minimal crooked ideal triangle\/} 
$\Tt(\Delta,O)$ as:
\[
\Tt(\Delta,O) := 
\big( \HH(\vs_0,O) \cup \HH(\vs_1,O) \cup \HH(\vs_2,O)\big)^c
\]
Then the faces of $\Tt(\Delta,O)$  are the crooked planes 
vertexed at $O$ and directed by spacelike vectors coresponding 
to the sides of $\Delta$. 

\subsection{Minimality and nondegeneracy}\label{sec:MinimalNondegeneracy}
Every crooked ideal triangle determines a minimal crooked ideal triangle.
To recover a nondegenerate crooked ideal triangle from its associated
minimal crooked ideal triangle, one attaches 
\emph{parallel crooked slabs} to its faces.  
The goal of \S\ref{sec:MinimalNondegeneracy} 
is the following Structure Theorem for crooked ideal triangles.

\begin{thm}\label{thm:StructureTheorem}
Let $\Tt$ be a crooked ideal triangle
with linearization $\Delta=\L(\Tt)$.
Then there is a unique point $O$ 
(the {\em center\/} of $\Tt$) 
such that
$\Tt$ contains  the minimal crooked ideal triangle
$\Tt(\Delta,O)$.
Futhermore $\Tt$ decomposes uniquely as the union of 
$\Tt(\Delta,O)$
with three parallel crooked slabs attached along the faces of 
$\Tt(\Delta,O)$.
\end{thm}
\noindent  Theorem~\ref{thm:StructureTheorem}
derives from the theory of crooked halfspaces, 
of which 
a detailed description may be found in 
\cite{burelle2012crooked}.

%
%

\subsubsection{From minimal to nondegenerate}
Start with a minimal crooked ideal triangle. 
To each crooked plane lying on the  boundary,  attach three
{\em parallel crooked slabs,\/}
defined as the regions between two \emph{parallel crooked planes}, that is, two crooked planes differing only by translation.

Let $\vs$ be a unit spacelike vector determining a halfplane in $\Ht$.
For $p_1,p_2\in\E$, 
\[ \HH(\vs,p_1) \supset \HH(\vs,p_2) \]
if and only if the vector $p_2-p_1$ lies in a 2-dimensional convex linear cone
in $\V$, called the 
{\em translational semigroup\/}  $\V(\vs)$ of $\vs$. 
By \cite{burelle2012crooked},
\begin{align}\label{eq:TranslationalSemigroup}
\V(\vs) & = \spn \{\vs^-, -\vs^+\} \notag \\
& = \{ \um \vs^- - \up \vs^+ \mid \um,\up > 0 \}.\end{align}
\begin{defn*}
A {\em  parallel crooked slab\/} is the complement
\[ {\mathsf{ParSlab}}(\vs; p_1,p_2) := 
\overline{\HH(\vs,p_1)} \setminus \HH(\vs,p_2) \]
where $\vs$ is a unit spacelike vector and $p_2 - p_1\in\V(\vs)$.
\end{defn*}
\noindent
Choose consistently oriented unit-spacelike vectors $\vs_0,\vs_1,\vs_2$ corresponding to 
the sides of an ideal triangle and an origin $O\in\E$. 
Then  the  open crooked halfspaces $\HH(\vs_i,O)$ are mutually
disjoint and asymptotic. 
The complement of their union
\begin{equation} \label{eq:MinCIT}
\Tt(\Delta,O) :=  \big(
\HH(\vs_0,O) \cup \HH(\vs_1,O) \cup \HH(\vs_2,O) \big)^c
\end{equation}
is a minimal crooked ideal triangle. 
The general crooked ideal triangle is the complement
\begin{equation} \label{eq:GenCIT}
\Tt\big(\Delta,(p_0,p_1,p_2)\big) :=  
\big(\HH(\vs_0,p_0) \cup \HH(\vs_1,p_1) \cup \HH(\vs_2,p_2) \big)^c
\end{equation}
\noindent
where the vector $p_i - O$ lies in $\V(\vs_i)$. 
Furthermore the crooked ideal triangle 
$\Tt\big(\Delta,(p_0,p_1,p_2)\big)$ is nondegenerate since
$p_i - O\in\V(\vs_i)$. 
It decomposes as a union of  the minimal crooked ideal triangle with three parallel crooked slabs attached along the faces:
\begin{align}\label{eq:CITdecomposition}
\Tt\big(\Delta,(p_0,&p_1,p_2)\big)\;  := \ \Tt(\Delta,O)\ \cup \notag \\
& \big({\mathsf{ParSlab}}(\vs_0; p_0,O)\cup{\mathsf{ParSlab}}(\vs_1; p_1,O)\cup{\mathsf{ParSlab}}(\vs_2; p_2,O)\big).
\end{align}
Call $\Tt(\Delta,O)$ a {\em minimal crooked ideal subtriangle.\/}
We show 
that every nondegenerate crooked ideal triangle arises in this way
(\S\ref{sec:Existence}),
and this decomposition is unique
(\S\ref{sec:Uniqueness}) .

\subsubsection{Existence of a minimal crooked ideal subtriangle}\label{sec:Existence}
Say that three linear planes
$\PP_0,\PP_1,\PP_2\subset\V$ are
{\em in general position\/} if and only if
$\PP_0 \cap \PP_1 \cap \PP_2 = 0$.
Equivalently, three planes are in general
position if they are respectively 
defined by three linearly independent covectors.

\begin{lemma}\label{lem:StemPlanes}
Let $\PP_0,\PP_1,\PP_2\subset\V$
be planes in general position.
Let $(p_0, p_1,p_2)\in \E^3$ be a triple
of points.
Then  %
\begin{align*}
p_0 & =  p +  \qv_0 \\
p_1 & =  p +  \qv_1 \\
p_2 & =  p + \qv_2.
\end{align*}
for a unique $p\in\E$  
and unique $\qv_0\in\PP_0$, $\qv_1\in\PP_1$, $\qv_2\in\PP_2$.
\end{lemma}
\noindent
The vertex triple $(\qv_0, \qv_1,\qv_2)$ is said to be {\em normalized\/} 
with respect to the three planes $\PP_0,\PP_1,\PP_2$ and the point $p$.

\begin{proof}
Choose an arbitrary origin $O\in\E$ to identify
$\E$ with the vector space $\V$. 
Then we identify triples
$(p_0, p_1,p_2)\in\E^3$ with a triples $(\qv_0, \qv_1,\qv_2)\ \in\  \in\V^3$.
Given this choice of $O$ we  show \[
(\qv_0, \qv_1,\qv_2)\ \in\  
\PP_0 \times \PP_1 \times \PP_2\ \subset\ \V\times\V\times\V
\]
for a {\em unique\/} vector $\vv\in\V$, such that
\begin{align*}
p_0 & =   \qv_0 + \vv \\
p_1 & =   \qv_1 + \vv \\
p_2 & =   \qv_2 + \vv 
\end{align*}
\noindent
The map 
\begin{align*}
\PP_0 \oplus \PP_1 \oplus \PP_2\oplus\V
&\longrightarrow \V \oplus \V \oplus \V \\
\qv_0 \oplus \qv_1 \oplus \qv_2 \oplus \vv &\longmapsto
(\qv_0 +\vv) \oplus (\qv_1 + \vv)  \oplus (\qv_2 + \vv) 
\end{align*}
is  a linear map between two $9$-dimensional vector spaces, 
with kernel defined by:
\[ (\qv_0 +\vv) =(\qv_1 + \vv)   =  (\qv_2 + \vv) = 0 \]
whence 
\[ \vv = -\qv_0 = - \qv_1 = -\qv_2 \in 
\PP_0 \cap \PP_1 \cap  \PP_2 = 0. \]
Thus the linear map is injective, and hence an isomorphism.
\end{proof}

\begin{lemma}
A crooked ideal triangle contains a minimal crooked ideal triangle.
\end{lemma}
\begin{proof}
Let $\Tt$ denote a crooked ideal triangle,  $\Delta = \L(\Tt)$, is its linearization  and 
$(\s{0},\s{1},\s{2})$ the triple of unit-spacelike vectors describing 
the sides of $\Delta$.  The vertex triple is written 
$(p_0,p_1,p_2)$,  so that $\Tt$ equals the complement 
\[ 
\Tt\ =\ \big(\HH(\s{0},p_0) \cup \HH(\s{1},p_1) \cup \HH(\s{2},p_2)\big)^c. 
\]
Apply Lemma~\ref{lem:StemPlanes} to the stem-planes
$\PP_i := \s{i}^\perp$ to normalize the vertex triple:
by finding a point $p\in\E$ such that $p_i - p$ lies in the
stem-plane $\s{i}^\perp$. Since the halfspaces 
$\HH(\s{i},p_i)$ are pairwise disjoint,
the vectors $p_i-p$ must lie in $\V(\s{i})$ by 
\disjointnesscriterion of \cite{burelle2012crooked}.
In particular $\Tt$ contains 
the minimal crooked ideal triangle $\Tt(\Delta,p)$ as desired.
\end{proof}

\subsubsection{Uniqueness of the minimal crooked ideal subtriangle}\label{sec:Uniqueness}
Now we show that the decomposition \eqref{eq:CITdecomposition} of 
a crooked ideal triangle $\Tt$ into a minimal crooked ideal triangle and
three parallel crooked slabs is unique. 


First, note that when one crooked ideal triangle lies inside another, 
they have the same linearization.

\begin{lemma}\label{lem:linearization}
Suppose $\Tt, \Tt' \subset \E$ are crooked ideal triangles
such that $\Tt\subset \Tt'$. 
Then $\L(\Tt) = \L(\Tt')$. 
\end{lemma}
\begin{proof}
Let $\H{i}, \H{i}'$ be the halfspaces complementary to 
$\Tt, \Tt'$ respectively.
Then, by hypothesis,
\[
\H{i}' \subset
\H{0} \cup \H{1} \cup \H{2} 
\]
for $i=0,1,2$. Since the $\H{j}$ are disjoint, each 
$\H{i}'$ lies in a unique $\H{{j(i)}}$. 
Re-index the $\H{{j(i)}}$ so that $j(i)=i$, that is,
\[
\H{i}' \subset \H{i}.
\]
By \orderpreserving of  
\cite{burelle2012crooked},
\begin{equation}\label{eq:hyperbolicinclusion}
\L(\H{i}')\subset \L(\H{i})
\end{equation}
for $i=0,1,2$. 

Denote the linearizations $\L(\Tt), L(\Tt')$ by
$\Delta,\Delta'$ respectively. 
The sides of $\Delta$ (respectively $\Delta'$) are
$\partial\Hh_i$ (respectively $\partial\Hh_i'$),
using the notation of \S\ref{sec:IdealTriangle}.
By \eqref{eq:hyperbolicinclusion},
each side $\partial\Hh_i'$ of $\Delta'$ is an ideal interval contained in
the side $\partial\Hh_i$ of $\Delta$.
As each of these triples of sides partition $\partial\Ht$, 
this implies that $\Delta = \Delta'$ as claimed.
\end{proof}
\begin{proof}[Conclusion of the proof of Theorem~\ref{thm:StructureTheorem}]
Let $\Tt$ be a crooked ideal triangle with linearization $\Delta$.
It remains to show that the decomposition \eqref{eq:CITdecomposition}
is unique. 
Suppose that  $\Tt(\Delta,O)$ and  $\Tt(\Delta',O')$ are minimal crooked ideal
triangles contained in $\Tt$; we must show that
$\Tt(\Delta,O) = \Tt(\Delta',O')$.

Lemma~\ref{lem:linearization} implies $\Delta' = \Delta$, so it remains to show
that the centers coincide: $O' = O$. This follows from the construction in 
Lemma~\ref{lem:StemPlanes}.
Observe that the sides of $\Delta$ determine the stem-planes 
$\PP_i = \vs_i^\perp$, and the vertex triple $(p_0,p_1,p_2)$ of $\Tt$ satisfies
\begin{align*}
p_i & \in \PP_i + O \\
p_i & \in \PP_i + O' 
\end{align*}
for $i=0,1,2$.
Uniqueness in Lemma~\ref{lem:StemPlanes} implies that $O' = O$  as claimed.
\end{proof}

\subsection{The deformation space of crooked ideal triangles}\label{sec:DefCIT}
Theorem~\ref{thm:StructureTheorem}  precisely
describes the moduli of crooked ideal triangles. 
First we review some well-known facts about ideal triangles in $\Ht$.

An ideal triangle $\Delta\subset\Ht$ is complementary to three disjoint
open halfplanes $\Hh(\s{0}), \Hh(\s{1}), \Hh(\s{2})\subset\Ht$,
whose bounding geodesics
$\partial\Hh(\s{i})$ are pairwise asymptotic. 
The defining unit-spacelike vectors form a basis
$(\s{0}, \s{1}, \s{2})$ of $\V$ with Gram matrix
\[
\bmatrix 1 & -1 & -1 \\ -1 & 1 & -1 \\ -1 & -1 & 1 \endbmatrix.
\]
The group $\SOto$ acts simply transitively on ideal triangles
(since it acts simply transitively on such bases). 

Fix an ideal triangle $\Delta$ as above. 
An origin $O\in\E$ determines a minimal crooked ideal triangle
$\Tt(\Delta,O)$ by \eqref{eq:MinCIT}.
An affine deformation of $\Delta$ is now determined by a
vertex triple $(p_0,p_1,p_2)\subset \E^3$,
given by \eqref{eq:GenCIT}.
By normalizing the vertex triple we find an origin such that 
\[
\Tt\big(\Delta,(p_0,p_1,p_2)\big) \supset \Tt(\Delta,O). 
\]
The vertex triple determines a triple of vectors 
\[ \qv = (\qv_0,\qv_1,\qv_2)\in \V^3 \] by: 
\begin{align}\label{eq:CITcondition}
\qv_i & := p_i - O  \notag \\
&= \um_i \vs_i^- - \up_i\vs_i^+\in \V(\vs_i),
\end{align}
that is, $\um_i,\up_i>0$ for $i = 0,1,2$.
Thus nondegenerate crooked ideal triangles
corresponds to pairs 
$\big(\Delta,(p_0,p_1,p_2)\big)$ as above,
or equivalently to triples 
$(\Delta,O, \qv)$ 
where 
\begin{align*}
\qv &=  (\qv_0,\qv_1,\qv_2) \\
& \in \V(\vs_0) \times \V(\vs_1) \times \V(\vs_2) \subset \V^3.
\end{align*}
The $\Iso$-equivalence classes form a moduli space
$\V(\vs_0) \times \V(\vs_1) \times \V(\vs_2)$
parametrized by triples 
$\qv\in \V^3$ satisfying \eqref{eq:CITcondition}.
This moduli space is an open orthant in $\R^6$, 
with coordinates
\[(\up_0,\um_0,\up_1,\um_1,\up_2,\um_2)\in \R_+^6.\]
Its quotient by homotheties $\R_+$ consists of $\Sim$-equivalence classes,
and identifies with a 
$5$-simplex in the double cover of projective $5$-space.

\section{Affine deformations}\label{sec:IdealTrianglesCoxeterGroups}
Two-generator discrete subgroups $\Gamma_0$ of $\SOto$ always extend
to Coxeter groups  $\Gamma_0^\e$  acting on $\Ht$. 
Similarly, an affine deformation $\Gamma$  of $\Gamma_0$  extends to an
affine deformation $\Gamma^\e$  for  $\Gamma_0^\e$: 
\[
\begin{CD}
 \Gamma^\e  @>>>  \Gamma_0^\e  \\
    @VVV  @VVV \\
   \Gamma @>>>  \Gamma_0\\
\end{CD}
\]
\noindent
The discrete Coxeter group   $\Gamma_0^\e$ admits  an ideal triangle 
as a fundamental polygon. 
Similarly, its proper affine deformation $\Gamma^\e$ may admit a crooked ideal triangle as a fundamental polyhedra.
Now we describe these Coxeter groups, 
their fundamental polygons and their affine deformations.
In particular, Lemma~\ref{lem:alpha} computes the Margulis-invariant  parameters
for these affine deformations. 
These relate to the parameters in the deformation space of crooked ideal triangles in a crooked fundamental domain, as descrbed in the previous section.

\subsection{Affine deformations of involution groups}
Affine deformations of $\Gamma_0^\e$ correspond to affine deformations
of the fundamental ideal triangle $\Delta$.
Suppose that $\Tt$ is a crooked ideal triangle with linearization
$\Delta$ and vertex triple $(p_0,p_1,p_2)$.
Let $[\t{0}], [\t{1}], [\t{2}]$ be three points on the respective sides of
$\Delta$ as above. For $j=0,1,2$ the unique representative $\t{j}$ 
of $[\t{j}]$ containing $p_j$ is a particle on the $j$-th face of $\Tt$.
The corresponding involutions $\tiota_j$ defined by \eqref{eq:ParticleInvolution}
generate an affine deformation $\Gamma^\e$ of $\Gamma_0^\e$. 

Suppose $\Tt$ is a nondegenerate crooked ideal triangle.
By~\cite{drummThesis,drumm1992fundamental,drumm1993linear} 
(and also \cite{MR1796126}), 
the affine deformation $\Gamma^\e$ acts properly on $\E$ with crooked fundamental
polyhedron $\Tt$. 
Affine isometry classes of these structures form a convex
set parametrized by  triples 
\[
\qv \ =\  (\qv_0,\qv_1,\qv_2)\ \in\ \V\times\V \times \V
\] satisfying \eqref{eq:CITcondition} (compare \S\ref{sec:DefCIT}).


The compositions $\tiota_j\tiota_{j+1}$ for $j=0,1,2\  ( \mod\  3)$,
are hyperbolic affine isometries (affine boosts):
\begin{align}\label{eq:AffCox}
\rho(A) &:= \tiot{2}\tiot{0}, \notag\\
\rho(B) &:= \tiot{0}\tiot{1}, \notag\\
\rho(C) &:= \tiot{1}\tiot{2},
\end{align}
and satisfy the relation $\rho(A)\rho(B)\rho(C) = 1$.
Denote the unit-spacelike neutral eigenvectors of their linear parts by:
\begin{align*}
\Ao & := \rho_0(A)^0 \\
\Bo & := \rho_0(B)^0 \\
\Co & := \rho_0(C)^0 \end{align*}

\begin{lemma}\label{lem:alpha}
The Margulis invariants of the affine deformation $\rho$ are:
\begin{align*}
\alpha\big(\rho(A)\big) &:= 2 (\qv_2 - \qv_0) \cdot \Ao \\
\alpha\big(\rho(B)\big) &:= 2 (\qv_0 - \qv_1) \cdot \Bo \\
\alpha\big(\rho(C)\big) &:= 2 (\qv_1 - \qv_2) \cdot \Co 
\end{align*}
\end{lemma}
\begin{proof}
Applying \eqref{eq:AffCox} to $p_0$, 
\[
\rho(A): p_0   \longmapsto \tiota_2(p_0) ,
\]
since $p_i\in \vt_i$ is fixed by $\tiota_i$.
Apply $\eqref{eq:involution}$ to $\tiota_2$ by noting that the
involution fixes the line through $p_2$ parallel to 
$\t{2}$ (set  $\vv= p_0 - p_2$ and $\vu = \t{2}$):
\begin{align}\label{eq:inv} 
 \tiota_2(p_0)  = & \tiota_2\big(p_2 + (p_0-p_2)\big) \notag \\
  & =  p_2 - (p_0 -p_2) + 
2 (p_0- p_2) \cdot \t{2} /(\t{2}\cdot \t{2})\;  \t{2} \notag \\
& \equiv 2 p_2 - p_0 \ (\mod\ \vt_2).
\end{align}
Now 
\[
\Axis\big(\rho(C)\big)\ \cap\ \Axis\big(\rho(A)\big)\ =\ \{\vt_2\}
\]
whence $\vt_2 \in\Axis\big(\rho(A)\big)$
implies that $\vt_2 \cdot \Ao = 0$. Thus: 
\begin{alignat*}{2}
\alpha\big(\rho(A)\big) & =
\big( \tiota_2(p_0) - p_0\big) \cdot \Ao && \qquad\text{~by~\eqref{eq:AffCox}}  \\
&  =
\big((2 p_2 - p_0)-p_0\big) \cdot \Ao 
&& \qquad\text{~by~\eqref{eq:inv}}  \\
&  =
\big(2 (p_2 - p_0)\big) \cdot \Ao.
\end{alignat*}
Now $p_2-p_0 = \qv_2 - \qv_0$ so
\[ 
\alpha\big(\rho(A)\big) = 2 (\qv_2-\qv_0)\cdot \Ao,
\]
as desired. The cases $\alpha\big(\rho(B)\big)$ and $\alpha\big(\rho(C)\big)$ are 
completely analogous. 
\end{proof}

\subsection{Affine deformations of ideal triangle Coxeter groups}
We have seen how to pass from an ideal triangle fundamental domain 
for the action of the Coxeter group to a fundamental quadrilateral of a
one-holed torus group.  
To begin the investigation of affine deformations of these groups, 
start again with  Coxeter groups and ideal triangles.

The vertices of the ideal triangle $\Delta$ correspond to fixed points
of commutators of basic pairs in $\Ft$ as follows.
\begin{align*}
\Kplus & = \xp{[A,B]} = [\vn] \\
\bm(\Kplus) & = \xp{[\bm,A]} = \iot{2}[\vn] \\
C(\Kplus) & =\xp{[\am,\bm]} = \iot{0}[\vn]
\end{align*}
where $\vn$ is defined as in \S\ref{sec:FundQuad}.
Define affine deformations 
$\tideal$ of the ideal triangle $\Delta\subset\Ht$,
which depend on a choice of a vertex triple
$(p_0, p_1,p_2)\in \E^3$
for a crooked ideal triangle $\ideal$
as in \S\ref{sec:CrookedIdealTriangles}.

\subsubsection{Choosing a vertex triple}
For the above ideal triangle $\Delta$, 
consider the three {\em stem-planes\/}
\begin{align*}
\s{1}^\perp &:= \R\big(\iot{2}\vn\big) + \R\big(\iot{0}\vn\big) \\
\s{2}^\perp &:= \R\big( \iot{2}\vn\big)  + \R\big( \vn\big) \\
\s{0}^\perp &:= \R\big(\vn\big) + \R\big(\iot{0}\vn\big).
\end{align*}
\noindent
Since $\vn, \iot{2}\vn, \iot{0}\vn$ base $\V$,
these three stem-planes are mutually transverse.
The translational semigroups 
\[ \V(\s{1})\subset \vs_1^\perp,\quad
\V(\s{2})\subset \vs_2^\perp,\quad
\V(\s{0})\subset\vs_0^\perp  \]
of the respective crooked halfspaces (as in \cite{burelle2012crooked})
are  quadrants:
\begin{align*}
\V(\s{1})  &:= \R^+ \big(-\iot{0}\vn\big) + \R^+\big(\iot{2}\vn\big) \\
\V(\s{2}) &:= \R^+\big(-\iot{2}\vn\big) + \R^+\big(\vn\big) \\
\V(\s{0}) &:= \R^+\big(-\vn\big) + \R^+ \big(\iot{0}\vn\big).
\end{align*}
Apply Lemma~\ref{lem:StemPlanes}  to  choose an origin $O\in\E$,
and vectors $\qv_0,\qv_1,\qv_2$ normalizing the vertex triple $p_i = O + \qv_i$:
\begin{alignat}{2}\label{eq:VertexTriple}
\qv_1&:= -\up_1\  (\iot{0}\vn)    &&+ \um_1\ (\iot{2}\vn) \notag\\
\qv_2 &:= -\up_2\  (\iot{2}\vn) &&+ \um_2( \vn) \\
\qv_0 &:= -\up_0\ (\vn)          &&+  \um_0\ (\iot{0}\vn) \notag 
\end{alignat}
where $\up_i,\um_i\in\R$.

The vertex triple $(p_1,p_2,p_0)$ defines three crooked halfspaces 
\[ \HH_i = \HH(\vs_i,p_i)\subset\E \]
where $\qv_i$ is the unit spacelike vector corresponding to the halfplane 
$\HI\subset\Ht$.

\begin{lemma}\label{lem:disjointhalfs}
When all the coefficients $\up_i,\um_i > 0$ (that is, when
$p_i\in\V(\s{i})$), then 
the crooked halfspaces $\HH_i$ are pairwise disjoint.
\end{lemma}
\begin{proof}
Apply the disjointness criterion~\cite{
burelle2012crooked,MR2653729,drumm1999geometry}
\end{proof}
\noindent
The {\em crooked ideal triangle\/} $\tideal$ is by definition the
intersection of the complements, or equivalently,
\[
\tideal := \E \setminus \big( \iint(\HH_0) \cup \iint(\HH_1) \cup \iint(\HH_2) \big). \]

\subsubsection{Calculating Margulis invariants}

\newcommand{\alphaABC}{\bmatrix \alpha(A) \\ \alpha(B) \\ \alpha(C) \endbmatrix}
\newcommand{\ABCmatrix}{\bmatrix 0 & \Ao & -\Ao \\-\Bo & 0 & \Bo \\\Co & - \Co & 0  \endbmatrix }
\newcommand{\qVector}{\bmatrix \qv_1\\\qv_2\\\qv_0\endbmatrix}
\newcommand{\rsMatrix}{\bmatrix 0 & -\up_1 & \um_1  \\ \um_2 & 0 & -\up_2 \\ -\up_0 & \um_0 & 0 \endbmatrix}
\newcommand{\nVector}{\bmatrix \vn \\ \iot{0}\vn  \\ \iot{2}\vn \endbmatrix}
\newcommand{\rsMatrixOne}{\bmatrix 0 & -\up_1 & \um_1  \\ 0 & 0 & 0 \\ 0 & 0 & 0 \endbmatrix}
\newcommand{\rsMatrixTwo}{\bmatrix 0 & 0 & 0 \\ \um_2 & 0 & -\up_2 \\ 0 & 0 & 0 \endbmatrix}
\newcommand{\rsMatrixThree}{\bmatrix 0 & 0 & 0 \\ 0 & 0 & 0 \\ -\up_0 & \um_0 & 0 \endbmatrix}

Now we compute an open cone of tame affine deformations in terms of the
respective $(\up,\um)$-coordinates of the respective vertices $p_1, p_2,p_0$:
\begin{prop}\label{prop:AlphaFromVertices}
Let $\up_i, \um_i\in \R$, for $i=1,2,0$, 
be the coefficients of a vertex triple
$(p_1,p_2,p_0)$ as in \eqref{eq:VertexTriple}.
Then 
\begin{equation}\label{eq:ThreeTerms}
\frac12 \alphaABC   = 
\MM_1 \bmatrix \up_1\\ \um_1 \endbmatrix + 
\MM_2 \bmatrix \up_2 \\ \um_2 \endbmatrix + 
\MM_0 \bmatrix \up_0 \\ \um_0 \endbmatrix 
\end{equation}
where
\begin{align}
\MM_1  &:=  
\bmatrix 0 \\ \Bo \\ -\Co \endbmatrix \cdot 
\bmatrix \iot{0}\vn & -\iot{2}\vn \endbmatrix, 
\notag \\
\MM_2  & := 
\bmatrix -\Ao \\ 0 \\ \Co \endbmatrix \cdot 
\bmatrix \iot{2}\vn & -\vn \endbmatrix, 
\notag \\
\MM_0  & := 
\bmatrix \Ao \\ -\Bo \\ 0\endbmatrix \cdot 
\bmatrix \vn & -\iot{0}\vn \endbmatrix.
\label{eq:MMatrices}
\end{align}
\end{prop}
\begin{proof}
Restate  Lemma~\ref{lem:alpha} as:
\[ \alphaABC= 2 \ABCmatrix\cdot \qVector\]
and  \eqref{eq:VertexTriple} as:
\[\qVector = \rsMatrix  \nVector\]
whence 
\begin{equation} \label{eq:AlphaVector}
\alphaABC = 2 \ABCmatrix\cdot \rsMatrix \nVector
\end{equation}
Decompose the last matrix as the sum of three matrices:
\[ 
\rsMatrixOne + \rsMatrixTwo + \rsMatrixThree. \]
For example:
\[ \ABCmatrix\cdot \rsMatrixOne \nVector = \MM_1 \bmatrix \up_1 \\ \um_1 \endbmatrix \]
The other two summands simplify similarly. 
Invoking the definitions \eqref{eq:MMatrices}, apply 
\eqref{eq:AlphaVector} to 
conclude the proof.
\end{proof}
\begin{prop}\label{prop:RankOne}
Each $3\times 2$-matrix $\MM_1,  \MM_2, \MM_0$ has rank one.
\end{prop}
\begin{proof}
We prove that 
\[ \MM_1 =
\bmatrix 0 & 0 \\ \Bo\cdot\iot{0}\vn & -\Bo\cdot\iot{2}\vn \\ 
-\Co\cdot\iot{0}\vn & \Co\cdot\iot{2}\vn \endbmatrix \] 
has rank one; $\MM_2$ and $\MM_0$ are handled similarly.
Since $\iot{0}\vn$ is not incident to $\Axis(B)$,
the inner product 
$\Bo\cdot\iot{0}\vn \neq 0$. Hence $\MM_1$ is nonzero, and 
its rank is positive. Thus it suffices to show that 
the $2\times 2$ submatrix
\[ \bmatrix  \Bo\cdot\iot{0}\vn & -\Bo\cdot\iot{2}\vn \\ 
-\Co\cdot\iot{0}\vn & \Co\cdot\iot{2}\vn \endbmatrix \] 
has determinant zero.
By \eqref{eq:CrossProductIdentity},
this determinant equals 
\begin{align*}
\big( \Bo\cdot\iot{0}\vn\big) &\big(\Co\cdot\iot{2}\vn\big) -
\big( \Bo\cdot\iot{2}\vn\big) \big(\Co\cdot\iot{0}\vn\big) \\ &=
(\Bo \times \Co)\cdot  ( \iot{0}\vn\times \iot{2}\vn) \end{align*}
which vanishes when the Lorentzian cross-product
$\Bo \times \Co$ is orthogonal to $\iot{0}\vn\times \iot{2}\vn$.
Now $\Bo \times \Co$ is a multiple of the timelike vector 
$\vt_1$ which represents $\Axis(B)\ \cap\ \Axis(C)$. 
The Lorentzian cross-product
$\iot{0}\vn\times\iot{2}\vn$ 
is a spacelike vector representing the geodesic having
endpoints $\iot{0}\vn$ and $\iot{2}\vn$. 
Since  this geodesic 
contains $\vt_1$, 
the vectors $\Bo \times \Co$ and $\iot{0}\vn\times\iot{2}\vn$ 
are orthogonal, and 
\[\big(\Bo \times \Co\big) \cdot \big(\iot{0}\vn\times \iot{2}\vn\big)=\ 0.\]
Thus the $2\times 2$ submatrix of $\MM_1$ is singular,
and the rank of $\MM_1$ equals $1$, as desired.\end{proof}

\subsection{Tiles and corners}
The vector space $\Ho$ parametrizes equivalence classes of affine deformations of 
$\Gamma_0$. 
As in \S\ref{sec:AlphaCoordinates}, this vector space has dimension $3$, and its sphere of directions
$\Sphere\Ho$ has dimension $2$.
Because of scale-invariance, it suffices to consider this
$2$-dimensional parameter space, and we regard the corresponding deformation spaces as subsets of
$\Sphere\Ho$. In particular we view the tiles 
(which were introduced  as $3$-dimensional simplicial ones) 
as triangular regions in $\Sphere\Ho$.
We view their edges 
(introduced as $2$-dimensional quadrants)
as line segments in $\Sphere\Ho$.

The superbasis defined by $(A,B,C)$ explicitly determines the isomorphism: 
\begin{align*}
\ho&\xrightarrow{\cong}\R^3 \\
[u] &\longmapsto \bmatrix \alpha(A)\\ \alpha(B) \\ \alpha(C)\endbmatrix.
\end{align*} 
The elliptic involution $\e_{A,B}$  inverts each 
basic element $A,B$.
Moreover $\e_{A,B}$ defines the Coxeter extension $\Gamma_0^\e$ of $\Gamma_0$.
The group $\Gamma_0^\e$  
has an ideal triangle $\Delta\subset \Ht$ 
as a fundamental domain. 
We study affine deformations of $\Gamma_0^\e$ which have  crooked ideal triangles
linearizing to $\Delta$ as fundamental domains.
By \S\ref{sec:IdealTrianglesCoxeterGroups},
such crooked ideal triangles are parametrized by a subset of 
the vector space 
\[ 
\s{0}^\perp\oplus\s{1}^\perp\oplus\s{2}^\perp\ \subset\ \V^3.
\]
Proposition~\ref{prop:AlphaFromVertices} describes the space of affine deformations
arising from nondegenerate crooked ideal triangles by the open $6$-dimensional orthant
\[
\V(\s{0}) \oplus \V(\s{1}) \oplus \V(\s{2}) \subset 
\s{0}^\perp\oplus\s{1}^\perp\oplus\s{2}^\perp. \]
\noindent
As above, affine deformations of the index two subgroup 
$\Gamma_0\subset\Gamma_0^\e$ are parametrized by the vector space
$\Ho\cong\R^3$.
Proposition~\ref{prop:RankOne} implies that the
image in $\Ho$ of each quadrant $\V(\s{i})$ is the ray  
\[  
\Corner_i :=  \MM_i\big(\V(\s{i})\big) \subset \ho \cong \R^3. 
\]
Proposition~\ref{prop:AlphaFromVertices} implies that the convex hull 
\[\mathfrak{T} := \Corner_0 + \Corner_1 + \Corner_2\]
parametrizes the affine deformations arising from a nondegenerate crooked
ideal triangle as in \S\ref{sec:IdealTrianglesCoxeterGroups}.

These affine deformations are all {\em proper.\/} 
The affine deformations of $\Gamma_0^\e$ admit a nondegenerate
crooked ideal triangle  $\Tt$ as a fundamental domain.
The affine deformations of $\Gamma_0$ admit a fundamental domain
$\Tt \cup \tiota_0(\Tt)$, and are therefore proper.

This deformation space  $\mathfrak{T} \subset \Ho$
is the  \emph{tile} referred to in Section~\ref{subsec:Super},
and parametrizes a set of {\em proper\/} affine deformations.
The tile is defined by its three  \emph{corners} $\CC_0$. $\CC_1$, and $\CC_2$. 
and the boundary is composed of its three  {\em edges:\/}
\begin{align*}
\Edge_0 &:= \CC_1 + \CC_2 \\
\Edge_1 &:= \CC_2 + \CC_0 \\
\Edge_2 &:= \CC_0 + \CC_1 \\
\end{align*}
where each pair of edges meet at a corner of the tile.

Affine deformations corresponding to corners are {\em never\/} proper
since their Margulis invariants vanish:
\begin{align*}
\alpha(A) = 0 &\text{~for affine deformations in~} \CC_1; \\
\alpha(B) = 0 &\text{~for affine deformations in~} \CC_2;\\
\alpha(C) = 0 &\text{~for affine deformations in~} \CC_0.
\end{align*}
However the edges parametrize {\em proper\/} affine deformations,
which admit crooked fundamental domains which are {\em crooked ideal quadrilaterals,\/}
unions of two {\em degenerate\/} crooked ideal triangles along a common face.

\subsection{Living on the edge}
The crooked fundamental domains for the Coxeter extension constructed from the ideal triangle $\ideal$ degenerate for parameter values on the edges.
The corresponding affine deformations admit crooked fundamental domains 
consisting of two degenerate crooked ideal triangles. 
These crooked fundamental domains are modeled on an ideal {\em quadrilateral\/} 
$\QQ$ decomposed into two adjacent ideal triangles in $\Ht$.

The construction of the tiles arose from ideal triangular domains of the Coxeter group. 
To show that the interior of the edges also correspond to proper deformations, 
we need crooked fundamental domains modeled on quadrilaterals. 

\subsubsection{Crooked ideal quadrilaterals}

The crooked ideal triangle $\Tt$ is bounded by three crooked planes
$\CP_i$ where $\HH_i := \HH(\vs_i,p_i)$, for $i=0,1,2$. 
Its reflected image $\tiota_0(\Tt)$
shares the face  $\CP_0$ with $\Tt$.
Its other two faces are
$\tiota_0(\CP_1)$ and $\tiota_0(\CP_2)$.
When the affine deformation approaches an edge,
the crooked ideal triangle $\Tt$ becomes degenerate, and the face
$\CP_0$ meets the other two faces $\CP_1,\CP_2$ of $\Tt$.

However, the face $\CP_0$ shields the reflected faces
$\tiota_0(\CP_1), \tiota_0(\CP_2)$ of $\tiota_0(\Tt)$
from the faces  $\CP_1, \CP_2$ of $\Tt$.
Since its faces are not disjoint, the degenerate crooked ideal triangle $\Tt$ is not a 
fundamental domain for $\Gamma^\e$.
However, the union $\Tt\cup\tiota_0(\Tt)$ is a crooked fundamental domain for
$\Gamma$, modeled on the ideal quadrilateral \[\iq = \Delta \cup \iota_0(\Delta). \] 

This crooked fundamental domain is a {\em crooked ideal quadrilateral,\/}
having faces
\[
\CP_1, \CP_2, \tiota_0(\CP_1), \tiota_0(\CP_2).\]

\subsubsection{Degnerating pairs of crooked ideal triangles}

Here is an explicit calculation, in coordinates, for the edge $\Edge_0$.
In that case,  $\up_0 = \um_0 = 0$, 
that is, the vertex of $\CP_0$ is the origin $O$.
Then \eqref{eq:VertexTriple} becomes:

\begin{alignat*}{2}
\qv_1 &:= \up_1 \big(\iot{0}\vn\big)   &&- \um_1\big( \iot{2}\vn\big) \notag\\
\qv_2 &:= \up_2 \big(\iot{2}\vn\big) && - \um_2 \big(\vn\big) 
\end{alignat*}
and the two hinges of $\CP_0$ are the photons
$\R\big(\iot{0}\vn\big)$ and $\R\big(\vn\big)$. 
While $\CP_2$ and $\CP_1$ remain disjoint, 
the stem of  $\CP_2$ meets the hinge $\R\big(\vn\big)$ of $\CP_0$ 
in the  past-pointing ray
\[ (-\um_2 - \Rplus) \vn \]
and the stem of  $\CP_1$ meets the hinge $\R \big(\iot{0}\vn\big)$ of $\CP_0$ 
in the future-pointing ray
\[ (\up_1 + \Rplus) \iot{0}\vn. \]
Similarly, the images of $\CP_1$ and $\CP_2$ under 
$\tiota_0$  are disjoint from each other, 
and disjoint from $\CP_1$ and $\CP_2$ respectively. 
Namely, 
\begin{alignat*}{2}
\iot{0}\qv_1 &:= \up_1 \big(\vn\big)   &&- \um_1\big(\iot{0} \iot{2}\vn\big) \notag\\
\iot{0}\qv_2 &:= \up_2 \big(\iot{0}\iot{2}\vn\big) && - \um_2 \big(\iot{0}\vn\big) 
\end{alignat*}
and  the stem of  $\iot{0}\CP_2$ meets the hinge $\R\big(\vn\big)$ of $\CP_0$ 
in the  future-pointing ray
\[ (\up_1 + \Rplus) \vn \]
and the stem of  $\CP_1$ meets the hinge $\R \big(\iot{0}\vn\big)$ of $\CP_0$ 
in the past-pointing ray
\[ (-\um_2 - \Rplus) \iot{0}\vn. \]
This produces the desired crooked fundamental ideal quadrilateral.


This procedure yields crooked fundamental ideal quadrilaterals for deformations 
parametrized by points on the edge $\Edge_0$. 
Deformations parametrized by points on $\Edge_1, \Edge_2$ are completely analogous.

\subsubsection{Disjointness of the tiles}
\begin{lemma}\label{lem:DisjointTiles}
Let $\bb,\bb'\in\BBB$ be neighboring superbases.
Then the corresponding tiles $\Tile_{\bb},
\Tile_{\bb'}$ are disjoint. 
\end{lemma}
\begin{proof}
We may assume $\bb$ and $\bb'$ correspond to basic triples
$(A,B,C)$ and $(B,A^{-1},C')$ respectively, where $C' := A B^{-1}$.
By applying a sign-change automorphism we may assume that
$\tr\big(\rho_0(A)\big), \tr\big(\rho_0(B)\big) > 0$, and by an elementary argument
(see \cite{MR2497777}), $\tr\big(\rho_0(C)\big) > 0$.  

Then 
\begin{align*}
\tr\rho_0(A) & = 2 \cosh(\ell_A/2) \\
\tr\rho_0(B) & = 2 \cosh(\ell_B/2) \\
\tr\rho_0(C) & = 2 \cosh(\ell_C/2) \\
\tr\rho_0(C') & = 2 \cosh(\ell_{C'}/2).
\end{align*}
Suppose given a deformation of $\rho_0$ whose derivative is a given cocycle which is the translational part of an proper affine deformation $\rho$. 
Then the derivative of the geodesic length function $\ell_A$ equals the Margulis invariant
$\alpha_A$ (as in Goldman-Margulis~\cite{MR1796129}. 
Differentiating the basic trace identity 
(as in Charette-Drumm-Goldman~\cite{charette2011finite})
\[
\tr\big(\rho_0(C')\big) = \tr\big(\rho_0(A)\big) \tr\big(\rho_0(B)\big) - \tr\big(\rho_0(C)\big)   \]
yields:
\begin{align*}
\sinh(\ell_{C'}/2) \alpha_{C'}  & = 
2\sinh(\ell_{A}/2) \cosh(\ell_{B}/2) \alpha_{A} \\  
& \quad+ 2\cosh(\ell_{A}/2) \sinh(\ell_{B}/2) \alpha_{B}  -
\sinh(\ell_{C}/2) \alpha_{C}.
\end{align*}
Thus $\alpha_C'$ is a linear combination 
$ a \alpha_A + b \alpha_B - c \alpha_C $ where $a,b,c > 0$.
This implies that the line defined by $\alpha_{C'} = 0$ intersects the 
closure of the triangular region defined by $\alpha_A, \alpha_B, \alpha_C \ge 0$
in the bounding lines $\alpha_A = 0, \alpha_B = 0$. 
\end{proof}
\begin{prop}\label{prop:DisjointTiles}
Let $\bb,\bb'\in\BBB$ be distinct superbases.
Then the corresponding tiles $\Tile_{\bb},\Tile_{\bb'}$ are disjoint. 
\end{prop}
\begin{proof}
Let $d$ be the natural metric on the tree $\T$ whose set of vertices equals $\BBB$.
We prove Proposition~\ref{prop:DisjointTiles} by induction on $d = d(\bb,\bb')$. 

The case $d=1$ corresponds to neighboring superbases,
which is Lemma~\ref{lem:DisjointTiles} by induction on $d$.

Now suppose $d>1$. Let
\[ \bb = \bb_1, \bb_2, \dots, \bb_d = \bb' \]
be the geodesic in $\T$ joining $\bb, \bb'$.
Inductively assume that $\bb_i\cap \bb_j = \emptyset$ if $\{i,j\}\neq\{1,d\}$.
The tiles are triangular subregions of the triangular region defined by
$\alpha_A, \alpha_B, \alpha_C \ge 0$.
Furthermore for any $i<d$, the edge separating $\Tile(\bb_i)$ and $\Tile(\bb_{i+1})$
extends to a line which separates $\Tile_{\bb},\Tile_{\bb'}$. 
Thus $\Tile_{\bb}\cap\Tile_{\bb'} = \emptyset$ as claimed. 
\end{proof}

%

\section{Proper affine deformations}
So far we have described two types of proper affine deformations
with crooked ideal polyhedra as fundamental domains.
The {\em tiles\/} correspond to 
proper affine deformations of Coxeter groups 
having a crooked ideal triangle as fundamental domain.
The {\em edges\/} correspond to affine structures having a crooked ideal
quadrilateral built out of two adjacent crooked ideal triangles 
(which are degenerate). 
Say that an affine deformation is {\em geometrically tame\/}
if is of one of these two types.
In this section we use the picture we have developed
to show that every {\em proper affine deformation\/}
is geometrically tame in this sense.

\subsection{The combinatorial structure}
 {\em Superbases}  were introduced in \S\ref{subsec:Super} as unordered triples of closed 
 curves whose pairwise intersection numbers are all $1$. 
 Given a superbasis $\bb = (X,Y,Z)$, 
 choose any ordering and find a corresponding  ordered set of isometries of the hyperbolic plane 
 $(A,B,C)$ such that $ABC=1$. 

Furthermore, there is extension of the group $\Gamma_0= \langle A, B, C \rangle$ generated
by three  involutions $(\iot{0}, \iot{1}, \iot{2} )$ such that   
\[ A = \iot{2}\iot{0},\  B=\iot{0}\iot{1},\  C= \iot{1}\iot{2}. \] 

A proper representation $\rho$ of $\Gamma_0$ extends to a representation of $\Gamma$ generated by the  affine involutions  $(\tiota{_0}, \tiota{_1}, \tiota{_2})$ 
defined by \eqref{eq:AffCox}. 
From the set of affine involutions, a tile was in $ \Sphere\Ho $ was constructed.
Thus, the superbasis $\bb\in\BBB$ corresponds to a tile in the  deformation space
\[\Tile(\bb)\subset\Sphere\Ho \]
where $\Sphere\Ho$ is the sphere of directions as defined in \S\ref{sec:scalings}.
Adjacent tiles meet along a common edge, 
and correspond to neighboring superbases. 
Recall that two superbases are adjacent if they contain a common
basis, and that an unoriented basis extends to  exactly  two superbases. 
Thus, every edge bounds exactly two tiles.
The graph whose $0$-simplices correspond to superbases and whose
$1$-simplices correspond to unoriented bases is the {\em trivalent tree\/}
dual to the curve complex of $\surf$. 
This is the {\em pants complex\/} of $\surf$.



\subsection{Geometric tameness}
Let $\Proper\subset\Sphere\Ho$ denote the space of proper affine deformations of a one-holed torus.Then 
\[
\Tame := \bigcup_{\bb\in\BBB} \overline{\Tile(\bb)}
\] 
parametrizes the set of geometrically tame affine deformations.
Since geometrically tame affine deformations admit crooked fundamental domains, they are proper affine deformations. 
\begin{thm}
 $\Tame = \Proper$.
\end{thm}
\begin{proof}
Since $\Tame\subset \Proper$,
it remains to show $\Proper \subset \Tame$.

We first show that  $\Tame$ 
is convex by induction, representing $\Tame$ as the increasing union
of a sequence of convex domains $\Tame_n$.  
Choose an initial superbasis $\bb_0$ and define
\[ \Tame_0 := \Tile(\bb_0) \]
which is a convex triangular region in $\ho$.

The three  superbases $\bb_1, \bb_2, \bb_3$ adjacent to $\bb_0$ 
determine triangular regions sharing edges with $\Tile(\bb_0)$. 
Define

\[ \Tame_1 :=  
\overline{\Tame_0} %
\cup \big( \Tile(\bb_1)\cup   \Tile(\bb_2) \cup  \Tile(\bb_3) \big).  
\]
This is a hexagon, and each side of its boundary $\partial\Tame_1$
lies inside the halfspaces bounded by $\alpha(X) \geq 0$ 
for some primitive element $X$ 
an element of the basis.
Since each angle of this $6$-gon is less than $\pi$, this $6$-gon is convex. 
Inductively define $\Tame_n$ as follows.
$\partial\Tame_n$ consists of $3\cdot 2^n$ edges. 
Each edge of $\partial\Tame_n$ is the edge of a triangle 
\[ \Tile(\bb_{\beta})\subset \Tame_n\] 
for some superbasis $\bb_{\beta}$. 
Constuct the polygon $\Tame_{n+1}$ by adding on the  triangle 
$\Tile(\bb_{\gamma})$ to each side $s \subset \partial\Tame_{n-1}$, 
where $\bb_{\gamma}$  is the superbasis neighboring
$\bb_{\gamma}$ such that 
\[ \Tile(\bb_{\gamma})\cup \Tile(\bb_{\beta}) = s. \] 
The two endpoints of $s$ lie on lines 
$\alpha(X_a) = 0$ and $\alpha(X_b) =0$ where 
\[ X_a,X_b = B_{\beta}\cap B_{\gamma}.\]
Thus all of the $\left( 3\cdot 2^{n+1}\right)$-gons  
$\Tame_{n+1}$ lie between the $\alpha(X)=0$ lines which touches its endpoints, 
so all of the angles of $\Tame_{n+1}$ must be less than $\pi$
 
The infinite union  $\Tame$ is also convex, 
since it is an increasing union of sharp convex domains.
To show this, for any $p_1 ,p_2\in U$ there must be some $\Tame_{n}$ 
containing these two points. 
Since each $\Tame_n$ is convex, their union $\Tame$ is convex.

Finally, every $p\in \Proper$ is also in $U$:
Assume otherwise, that is, that some $p\in \Proper \setminus \Tame$. 
However, because  $\Tame$ is convex, 
some line $L$ through $p$ does not intersect $\Tame$. 
Moreover, one of the halfplanes with boundary $L$ does not meet 
$\Tame$, but does meet $\partial\Proper$.
This subset of $\partial\Proper$ must contain
at least parts of two different segments of  $\partial\Tame$. 
Between any two segments of $\partial\Tame$ 
there must be an entire segment of $\partial\Tame$, 
so there must be an entire segment of $\Tame$ 
which lies in the complement of the closure of $\Tame$. 
This is a contradiction, as every segment of $\Tame$ contains some vertex of some triangle in $\Tame$.

\end{proof}

\begin{figure}
\centerline{\includegraphics
[scale=0.6]
{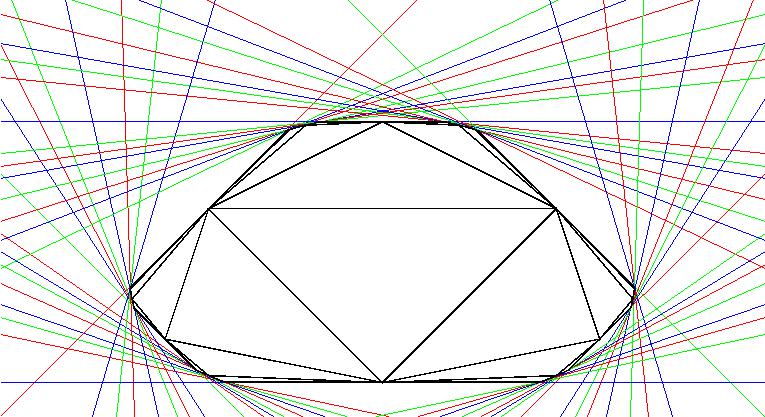}}
\caption{Tiling of the convex body of
proper affine deformations of a hyperbolic one-holed tori. 
The complete lines outside the convex body are the lines defined by the
covectors $\alpha(W)$, where $W$ ranges over nonseparating curves}
\end{figure}
%
%

 \bibliographystyle{amsplain}
 \bibliography{ranktwo.bbl}

\end{document}